\documentclass[pdflatex,sn-basic,margin = 1in]{sn-jnl}
\usepackage[utf8]{inputenc}
\usepackage[T1]{fontenc}
\usepackage{latexsym,amssymb,amsfonts,amsmath,amsthm,mathrsfs}
\usepackage{graphicx,float,xcolor,textcomp,bigints,mathdots}
\usepackage{longtable,tabularx}
\usepackage[title]{appendix}
\usepackage{hyperref}
\numberwithin{equation}{section}
\theoremstyle{thmstyleone}
\newtheorem{theorem}{Theorem}[section]
\newtheorem{lemma}[theorem]{Lemma}
\theoremstyle{thmstyletwo}
\newtheorem{remark}[theorem]{Remark}

\newcommand{\pg}[1]{{\color{black}  #1}}
\hypersetup{hidelinks}
\setcitestyle{authoryear,round,semicolon,aysep={},yysep={,}}
\begin{document}

\title{Symmetry-breaking in a discrete-choice LQG mean-field game}

\author[1]{\fnm{Ali Akbar Rezaei} \sur{Lori}}
\author*[1]{\fnm{Piyush} \sur{Grover}}\email{piyush.grover@unl.edu}
\affil[1]{\orgdiv{Department of Mechanical and Materials Engineering},
\orgname{University of Nebraska--Lincoln},
\orgaddress{\city{Lincoln}, \state{NE}, \country{USA}}}

\abstract{Mean-field games provide a continuum framework for modeling the dynamics of large, interacting populations of non-cooperative agents. This paper studies symmetry-breaking in a finite-horizon, two-choice min-Linear Quadratic Gaussian (LQG) mean-field game in which identical agents with linear stochastic dynamics choose one of two equally desirable terminal destinations, while trading off control effort against social pressure to conform. The model has an odd symmetry between the two destinations and therefore always admits a symmetric, dynamic Nash equilibrium in which the population splits evenly between them, producing a deadlock collective state at final time. Numerical studies have suggested that, as the penalty for social nonconformity increases, this symmetric equilibrium loses stability as a fixed point of an associated scalar, self-consistency map, and asymmetric consensus Nash equilibria emerge, where most agents select the same destination. By analyzing the linearized forward-backward system of partial differential equations through the scalar map representation, we provide a proof of this loss of stability of the symmetric equilibrium. Together with the odd symmetry of the map, this implies the existence of symmetry-broken mean-field game equilibria corresponding to consensus on either destination.}

\maketitle


\section{Introduction}
Dynamic collective choice problems involve decision-making agents whose objectives are coupled through aggregate effects, and where agents individually choose an option between multiple alternatives or destinations. Such collective decision making is routinely observed in nature, from animal groups that choose a nest site, food source or migrating direction \citep{seeley2001nest,ward2008quorum,wendt2020negative} to molecular switches that select between cell fates \citep{oppenheim2005switches}. 
By studying such systems, researchers have put forth several bio-inspired mathematical frameworks for distributed multi-agent decision making, with potential applications to opinion prediction, robotics, transportation networks and power systems \citep{lu2018multiple,march2024honeybee,salhab2018mean}.  

These mathematical frameworks of collective choice come in several different flavors. Some models treat choice as a static, one-shot decision \citep{brock2001discrete}, while others involve a dynamic system which models how the agents move, gather information, and eventually commit to an alternative \citep{salhab2017dynamic}. A second distinction is between agent-based descriptions, in which each individual is modeled explictly \citep{list2009independence}, and mean-field models, in which the effect of a large population is represented through an aggregate effect \citep{seeley2001nest}. Finally, compared to the deterministic models \citep{pais2013mechanism}, the stochastic models explictly account for the effect of unmodeled dynamics via persistent, individual-level noise \citep{list2009independence}. These modeling choices lead to different analytical structures, ranging from static noncooperative games and agent-based models to Markov decision processes and mean-field games (MFGs). 

The present work is concerned with the analysis of a MFG model of dynamic collective choice, proposed by Salhab et al. \citeyearpar{salhab2019collective}. MFGs are well suited to this setting and provide a continuum framework for studying large populations of interacting strategic agents, where each agent responds to the aggregate behavior of the population. The MFG theory was introduced independently by Huang, Caines, and Malhamé \citeyearpar{huang2006large} and by Lasry and Lions \citeyearpar{lasry2007mean}. Mathematically, MFG solutions are typically characterized by a coupled forward-backward system of nonlinear PDEs: a Hamilton-Jacobi-Bellman equation for the representative agent's value function and a Fokker-Planck equation for the population density. Although certain `monotonicity' conditions ensure uniqueness of MFG solutions \citep{bardi2019non}, a broad range of MFG models have been shown to admit multiple solutions and bifurcations \citep{yinmehmeysha12,kolokoltsov2017mean,grover2018mean,cirant2019existence,cirant2018variational,ullmo2019quadratic,hongler2020mean,grover2024phase,lori2025topological,knochenhauer2025long,hofer2025synchronization}.

In the `min-LQG' MFG formulation of dynamic collective choice \citep{salhab2019collective} that is the focus of the present work, a large population of agents evolves dynamically, while each agent minimizes a cost consisting of three competing contributions: a control cost, a terminal cost that encourages proximity to one of the discrete alternatives, and a social effect term that penalizes disagreement with the population. A key result of \citet{salhab2019collective} is that MFG equilibria can be characterized through the fixed points of a scalar-valued map. The authors further showed numerically that, for sufficiently large social effect, the system admits multiple Nash equilibria: in addition to the symmetric deadlock equilibrium, two asymmetric equilibria appear, corresponding to collective preference for one alternative over the other. 

The appearance of two asymmetric equilibria from a symmetric deadlock state is closely related to the pitchfork mechanisms that organize many finite-dimensional models of collective decision making \citep{gray2018multiagent,march2026symmetry}. In particular, Franci et al. \citeyearpar{franci2015realization} showed that increasing a social effect parameter can destabilize a deadlock state and produce two stable consensus states through a pitchfork bifurcation in networked opinion dynamics. Follow-up work by Gray et al. \citeyearpar{gray2018multiagent} used these results to design multi-agent decision-making dynamics with pitchfork structure inspired by honeybee nest-site selection. More general equivariant bifurcation-theoretic results have also been developed for multi-agent, multi-option decision-making systems, where symmetry-breaking branches correspond to consensus, deadlock, or dissensus patterns \citep{franci2022breaking}. These works provide a finite-dimensional dynamical systems explanation for how symmetric indecision can be broken by social interaction.

The MFG setting considered here has a similar symmetry-breaking phenomenology, but the analytical problem is substantially different. Rather than a finite-dimensional autonomous ODE for agent states, the equilibrium is characterized by a finite-time forward-backward PDE system. Thus, the relevant bifurcation parameter does not enter a closed finite-dimensional vector field directly; rather, it changes the solution of a coupled HJB-FP system. The main objective of this paper is to make this analogy rigorous at the level of the MFG PDEs.

Our analysis proceeds through the scalar self-consistency map introduced in \citet{salhab2019collective}. For each prescribed terminal population split, represented by the share $v\in[0,1]$ of agents choosing the first alternative, the associated optimal control problem and Fokker--Planck equation determine a new share $G(v;q)$ choosing that alternative. Here, $q$ denotes the strength of the social effect term, and the MFG equilibria correspond precisely to fixed points of $v\mapsto G(v;q)$. The symmetry of the two alternatives implies that $G(1-v;q)=1-G(v;q),$
so that $v^*=0.5$ is always a fixed point, corresponding to the symmetric deadlock equilibrium. 

We analyze the linearized map about the symmetric fixed point. This linear map encodes the solution of the linearized HJB-FP system around the symmetric MFG equilibrium and quantifies how perturbations in the prescribed terminal population split are amplified or suppressed. We derive a non-negative integral representation of the linearized map, and use it to obtain a monotonic in $q$ lower bound on the linearized gain of the self-consistency map. This implies that the symmetric fixed point is stable for sufficiently small social effect and loses stability when the social effect is large enough. Using the interval structure of the map and its symmetry, we then show that this loss of stability implies the existence of two additional fixed points, corresponding to symmetry-broken consensus MFG equilibria. We emphasize that the stability considered here is the fixed-point stability of the scalar map, and \emph{not} the temporal stability of the evolution resulting from the forward–backward MFG system.

In the following, the min-LQG MFG formulation is reviewed in Section \ref{sec:minlqg}. In Section \ref{sec:Gmap}, we present some analytical and numerical results for scalar map $G$. In Section \ref{sec:linearmap}, the characterization of the linearized map around the symmetric fixed point is presented. The main stability analysis and symmetry breaking results are discussed in Section \ref{sec:stab}. 
\section{The min-LQG discrete choice MFG problem}\label{sec:minlqg}
In this section, we review the min-LQG MFG formulation developed in Salhab et al. \citeyearpar{salhab2019collective}, specialized to 1D state space with identical agents, and two choices for final destination. The authors consider a population of $N$ cost-coupled selfish agents, where the noisy single integrator dynamics of the $i$th agent are given by
\begin{equation}
    \label{SDE}
    dx^i=(ax^i(t)+b u^i(t))dt+\sigma dw^i,
\end{equation}
where the real scalars $x^i(t)$ and $u^i(t)$ are respectively the $i$th agent's state and control. The $w^i(t)$ are independent Brownian motions with diffusion coefficient $\sigma$. Each agent minimizes the cost \begin{align}
    \label{cost fnc}J(u^i)=\mathbb E\{\int_0^T(\frac{q}{2}(x^i-\bar X)^2+\frac{r}{2}(u^i)^2)dt+\min_{j=1,2}\frac{C}{2}(x^i(T)-y_j)^2\},
\end{align}  
with $X=\dfrac{1}{N}\sum\limits_{i=1}^N x^i$, $q\ge0$, and $r>0$. The running cost incentivizes the agents to stay close to each other while using minimal control, and the terminal cost drives the agents to steer their final state close to one of the two final destination choices, $y_1$ and $y_2$. \pg{While our results generalize in straightforward manner to other parametric regimes, for simplicity of presentation we assume $a>0,b>0$ and $C>\dfrac{ar}{b^2}$.}

The MFG theory \citep{bensoussan2013mean} applied to this system yields that in the infinite population limit, the representative agent’s value function $V$ solves the following HJB equation
\begin{equation}\label{HJB}
   \partial_tV+\dfrac{\sigma^2}{2}\partial_{xx}V+a\,x\,\partial_xV+\frac{q}{2}(x-X)^2-\dfrac{b^2}{2r}(\partial_xV)^2=0,
\end{equation}
with terminal condition $V(T,x)=\min\limits_{j=1,2}\frac{C}{2}(x-y_j)^2$. Here the mean $X(t)=\bigintsss xf(t,x)dx$, where $f(x,t)$ is the agent density. The optimal feedback control is given by 
\begin{align*}
    u^*(t,x)=
    \begin{cases} 
      -\dfrac{b}{r}\partial_xV(t,x) & \text{for } t \in [0,T), \\
      0 & \text{for } t=T.
\end{cases}
\end{align*}  Each agent follows the optimal control law \( u^*(t,x) \), and agent density is governed by the Fokker-Planck equation with the optimal drift
\begin{equation}\label{FP f*_v}
    \partial_tf-\dfrac{\sigma^2}{2}\partial_{xx}f+\partial_x\big((ax+bu^*)f\big)=0,
\end{equation}
and a prescribed initial condition $f(0,x)=f_0(x)$.

\subsection{Solution of the Min-LQG MFG}\label{subsec:map}
Assuming a \emph{prescribed} mean trajectory $X(t)$, the optimal control $u^*(x,t)$ is shown to be a convex combination of the optimal controls $u_j(x,t)$ for the two standard LQG MFG problems obtained by using terminal conditions $V(T,x)=\frac{C}{2}(x-y_j)^2$ in Eqs. (\ref{SDE},\ref{cost fnc}) for $j=1,2$. At each time instance, an agent can choose one of the two $u_j$s as its control policy. 

\begin{align}\label{u*}
    u^*(t,x)=-\dfrac{b}{r}\partial_xV(t,x)=\sum_{j=1}^2\dfrac{\exp({-\eta V_j(t,x)})g_j(t,x)}{\sum\limits_{k=1}^2\exp({-\eta V_k(t,x)})g_k(t,x)}u_j(t,x)=\sum_{j=1}^2 \gamma_j(t,x)u_j(t,x),
\end{align}
with weights 
\begin{align}    
\gamma_j(t,x)=\dfrac{\exp({-\eta V_j(t,x)})g_j(t,x)}{\sum\limits_{k=1}^2\exp({-\eta V_k(t,x)})g_k(t,x)},\label{eq:gamma0} 
\end{align}
where $0\leq \gamma_j\leq 1$ \pg{and $\eta=\dfrac{b^2}{r\sigma^2}$}. 
Here, $\gamma_j(t,x)$ represents the instantaneous choice weight of control $u_j$ for an agent at state $x$ and time $t$. Assuming $y_1<y_2$, let $W_1\triangleq  (-\infty,\dfrac{y_1+y_2}{2})$ be the set of points on the real line closer to $y_1$ than $y_2$, and $W_2$ be its complement. Then, as $t\rightarrow T$, the weights $\gamma_1\rightarrow \mathbf{1}_{x\in W_1}$ and $\gamma_2\rightarrow \mathbf{1}_{x\in W_2}$, i.e., close to terminal time, each agent increasingly commits to the control associated with the destination that is closest to its current state.

The LQG MFG theory states that the value function $V_j$ and control $u_j$ are quadratic and linear functions of the state, respectively: 
\begin{equation}
    \label{Vj}
    \begin{split}
        &V_j(t,x)=\frac{1}{2}p(t)x^2+s_j(t)x+d_j(t), \\&
        u_j(t,x)=-\frac{b}{r}(p(t)x+s_j(t)).
    \end{split}
\end{equation}

The Riccati parameters $(p,s_j,d_j)$ satisfy the following ODEs,
\begin{equation}
    \label{hjb riccati2}
    \begin{split}
        &\dot{p}=\frac{b^2}{r}p(t)^2-2ap(t)-q, \\& \dot{s}_j=(\frac{b^2}{r}p(t)-a)s_j(t)+qX(t), \\&
        \dot{d}_j=\frac{b^2}{2r}(s_j(t))^2-\frac{\sigma^2}{2}p(t)-\frac{1}{2}qX(t)^2,
    \end{split}
\end{equation}
with final conditions $p(T)=C$, $s_j(T)=-Cy_j$, and $d_j(T)=\dfrac{C}{2}y_j^2$. 

In Eq. \ref{u*}, $g_j(t,x)$ is the conditional probability of an agent following the control policy $u_j$ to be in $W_j$, given that its state at time $t$ is $x$, 
\begin{align} \label{eq:gj}
    g_j(t,x)=\mathbb{P}(x_j(T)\in W_j|x_j(t)=x, u=u_j)=\frac{1}{\sqrt{2\pi\Sigma(t)}}\int_{W_j}\exp{\bigg(-\frac{(y-X_j(t) )^2}{2\Sigma(t)}\bigg)}dy,
\end{align} where $x_j(t)$ is the solution of the SDE \ref{SDE} with $u^i=u_j$, and initial condition $x^i(t)=x$. The mean $X_j$ and variance $\Sigma$ of the process are given by \begin{equation}\label{Xj Sigma}
    \begin{split}
        &X_j(t)= x\alpha(T,t)-\int_t^T \alpha(T,\tau)\frac{b^2}{r}s_j(\tau)d\tau, \\& \Sigma=\int_t^T\alpha^2(T,\tau)\sigma^2\,d\tau,
    \end{split}
\end{equation}
 where $\alpha(t,t_0)$ is the fundamental solution governed by 
 \begin{align}\label{eq:alphadot}
     \dot{\alpha}(t,t_0)=(a-\dfrac{b^2}{r}p(t))\alpha(t,t_0),\end{align} and initial condition $\alpha(t_0,t_0)=1$.

The ODEs for $p(t)$ and $\alpha(T,t)$ can be explicitly solved to obtain
\begin{equation}\label{alpha p(t)}
p(t)=\dfrac{p_1(q)-K(q)e^{-2\Delta(q)(T-t)}p_2(q)}{1-K(q)e^{-2\Delta(q)(T-t)}}, \qquad \alpha(T,t)= \dfrac{1-K(q)}{e^{\Delta(q)(T-t)}-K(q)e^{-\Delta(q)(T-t)}} ,
\end{equation}
where \begin{align} \label{eq:delta}
\Delta(q) = \sqrt{a^2 + b^2q/r},\end{align} 
$p_1(q) = (r/b^2)  (a + \Delta(q)),
p_2(q) = (r/b^2)  (a - \Delta(q)),
$ and 
\begin{align}\label{eq:K(q)}
K(q)  = (C - p_1(q)) / (C - p_2(q))\end{align}
with $|K(q)|<1$.

The optimal control described so far assumes a prescribed mean trajectory. Let $v$ denote the fraction of agents that end up in $W_1$ at time $T$. As shown in \citet{salhab2017dynamic}, this requires the unique mean trajectory to be
\begin{equation}\label{X mean}
    X(t;v)=R_1(t,0)X(0)+R_2(t)(vy_1+(1-v)y_2),
\end{equation}
where the functions $R_1$ and $R_2$ solve respectively the ODEs
\begin{equation}\label{R1}
    \begin{split}
            &\dot{R}_1(t,t_0)=(a-\frac{b^2}{r}p_m(t))R_1(t,t_0),
    \end{split}
\end{equation}
and 
\begin{equation}\label{R2}
        \dot{R}_2=(a-\frac{b^2}{r}p_m(t))R_2+\frac{b^2}{r}R_1(T,t)C,
\end{equation}
with $R_1(t_0,t_0)=1$, and $R_2(0)=0$.
Here $p_m$ is the solution of the `mean' Riccati equation 
\begin{equation*}\label{pv}
        \dot{p}_m=\frac{b^2}{r}p_m^2-2ap_m,
\end{equation*} with final condition $p_m(T)=C$. 

The optimal control and density pair $(u=u^*,f=f^*)$ satisfies the Fokker-Planck Eq. \ref{FP f*_v}. The standard MFG consistency condition requires that density evolution obtained from Eq. \ref{FP f*_v} reproduce the prescribed mean trajectory $X(t;v)$ used to compute the optimal control $u^*$ (using Eq. \ref{u*}), i.e., let $X^*(t;v,q)\triangleq$ $\bigintsss x f^*(x,t;v,q)dx$, then $X^*(t;v,q)=X(t;v)$.

A key result of \citet{salhab2019collective} establishes a one-to-one mapping between the dynamic MFG equilibrium $(u^*,f^*)$ that solves Eqs. \eqref{HJB} and \eqref{FP f*_v}, and the fixed points of a scalar-valued map, $G(v;q)$. This map is introduced and analyzed in the next section.

\section{Properties of the scalar map $G$}\label{sec:Gmap}
We write the self-consistency map as $G(\cdot;q):[0,1]\mapsto [0,1]$. For fixed $q$, it is given by the composition
$G(v;q)=G_1\circ G_2\circ G_3\circ G_4(v),$
where we suppress the dependence of the maps on $q$, see Fig. \ref{fig:Gmap_flow}.
The constituent maps are
\begin{equation}
\label{eq:maps}
\begin{aligned}
G_4 &: v \mapsto X(\cdot),\\
G_3 &: X(\cdot) \mapsto u^*(\cdot,\cdot),\\
G_2 &: u(\cdot,\cdot) \mapsto f(\cdot,\cdot),\\
G_1 &: f(\cdot,\cdot) \mapsto \tilde{v},
\end{aligned}
\end{equation}
where $G_4$ is defined by Eqs. (\ref{X mean}\textendash\ref{R2}). From properties of Riccati equations, $X\in C^\infty([0,T])$.  $G_3$ is defined by Eqs. (\ref{u*}\textendash\ref{eq:alphadot}), and it was established in \citet{salhab2019collective} that $u\in C^\infty([0,T)\times\mathbb R)$, and $|u(t,x)|\leq L(1+|x|)$. Thus, $u$ is smooth (except at terminal time) and grows at most linearly in $x$. $G_2$ is the solution map of Eq.\eqref{FP f*_v}, and \pg{$f\in C^{\infty}([0,T)\times\mathbb R)$} with finite second moments, by standard Fokker-Planck theory for the smooth drifts with at most linear in $x$ growth \citep{bogachev2022fokker}. Finally, $G_1(f)=\bigintssss_{-\infty}^{(y_1+y_2)/2}f(T,x)dx$. 

\begin{figure}[hbt!]
    \centering
    \includegraphics[width=0.95\linewidth]{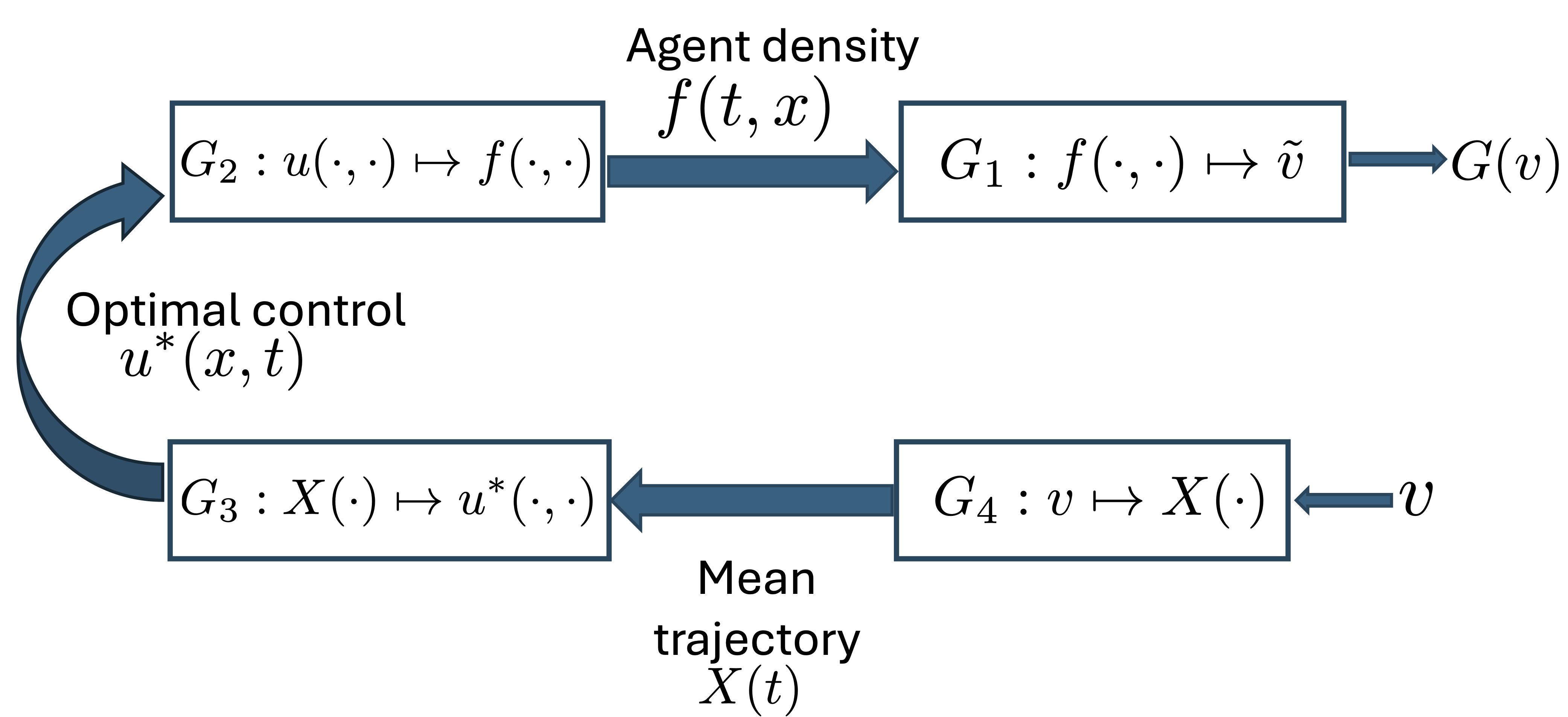}
    \caption{\footnotesize{The scalar map $G$ as a composition of four maps.}}
    \label{fig:Gmap_flow}
\end{figure}

For the rest of the paper, we focus on the case with symmetric destination choices, i.e., $y_2=-y_1$, a symmetric initial condition given by the unit zero-mean Gaussian, and positive drift $a>0$.

\subsection{Fixed point iterations and numerical results}
\begin{figure}[hbt!]
    \centering
    \includegraphics[width=0.95\linewidth]{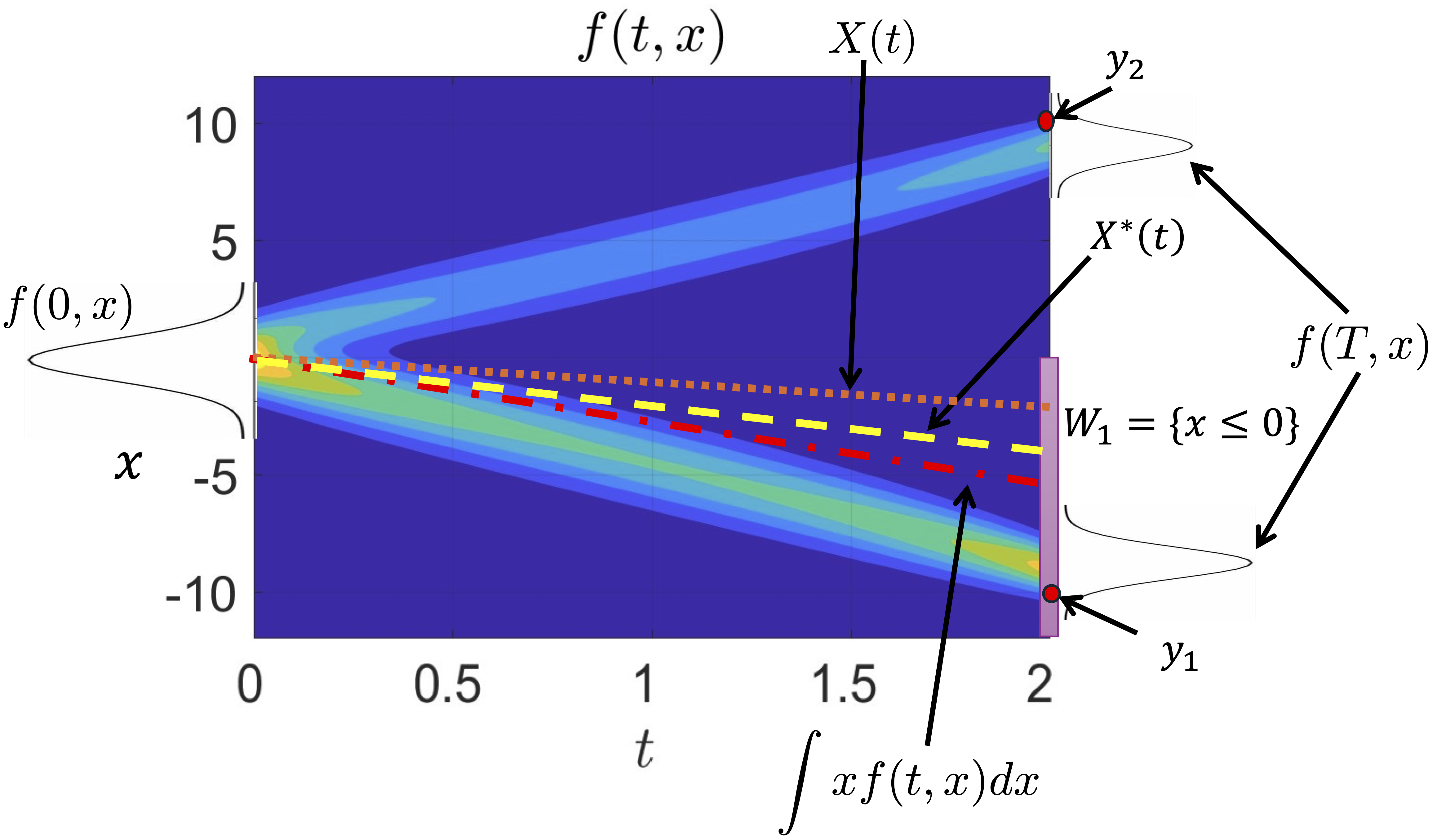}
    \caption{\footnotesize{An iteration of the map $G$ at $q=15$, with $y_1=-10$ and $y_2=10$ (red dots) the two choices at the final time. The parameters $a, b, r, C, \sigma$ and $T$ are respectively $0.1, 0.2, 5, 500, 1.5$ and $2$. With a randomly chosen $v$, we compute the mean $X(t)=G_4(v)$, shown in orange. The optimal control with this prescribed mean is $u^*(t,x)=G_3(X)$. The background colormap represents the density $f^*(t,x)=G_2(u^*$) obtained by propagating the initial zero-mean Gaussian, and the resulting mean $\int x f(t,x;v)dx$ is shown in red. Since  $\int x f(t,x;v)dx\neq X$, the chosen value of $v$ is not a fixed point, i.e., the first iterate $\tilde{v}=G(v)=G_1(f)\neq v$. Also shown is the fixed-point mean $\int x f(t,x;v^*)dx$ in yellow, corresponding to an asymmetric fixed point $v^*_-=0.24$.}}
    \label{map_itr_k}
\end{figure}
\pg{A fixed-point iteration of the map $G$, starting from an arbitrarily chosen $v$ is shown in Fig. \ref{map_itr_k} for $q=15$. We employ fixed-point iterations and a bisection method \citep{salhab2017dynamic} to compute the other two fixed points $v^*_{\pm} = \{0.76, 0.24\}$ at this value of social effect. The mean $X^*$ corresponding to the \pg{asymmetric fixed point $v^*_-=0.24$} is also shown. In Fig. \ref{fig:numerics}, we present further numerical results corresponding to that fixed point.}

\begin{figure}[hbt!]
\centering
\begin{tabular}{c c} \\ \hspace{-0.0in}
\includegraphics*[width=0.49\linewidth]{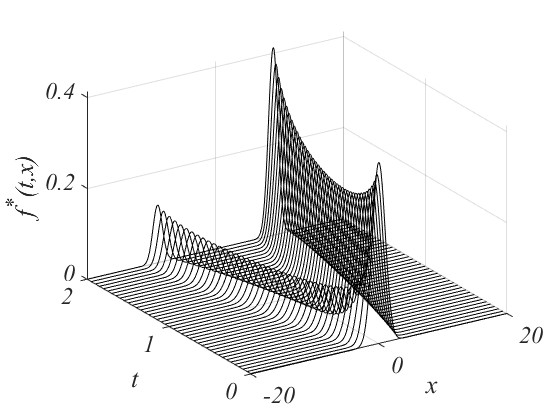} &\hspace{-0.4in} \includegraphics*[width=0.49\linewidth]{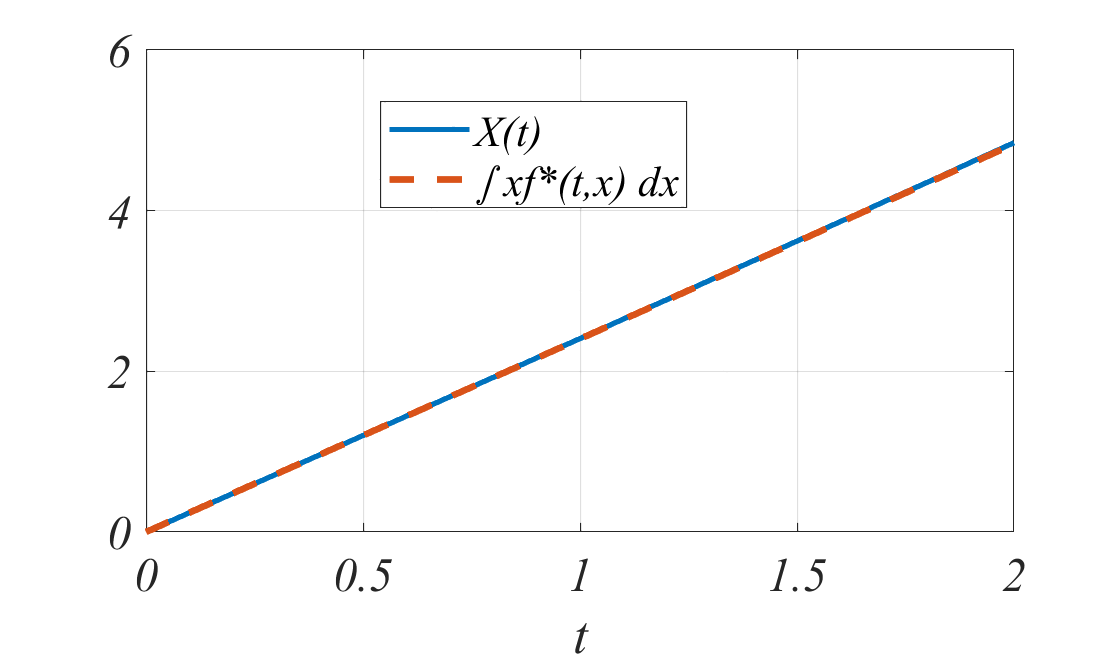} 
\\ (a) & (b)\\ \hspace{-0.2in}
\includegraphics*[width=0.49\linewidth]{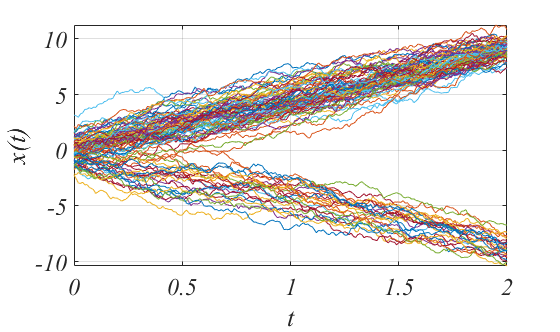} & \hspace{-0.2in}
\includegraphics*[width=0.49\linewidth]{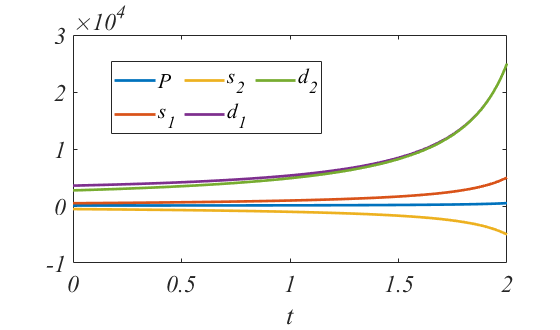} \\ (c) & (d)\\
\end{tabular}

\caption{ \footnotesize{Numerical results for the fixed point $v^*_-=0.24$ at $q=15$, corresponding to MFG solution pair $(f^*,u^*)$.
 a) The density evolution $f^*(t,x)$  b) Prescribed mean $X(t)$ from Eq.~\eqref{X mean} overlaid with the mean $\int x f^*(t,x)dx$ obtained from the corresponding density c) Trajectories of $N=401$ of agents using optimal control $u^*$, and d) Corresponding Riccati parameters.}}
\label{fig:numerics}
\end{figure}

\subsubsection{Numerical results on symmetry breaking via a pitchfork bifurcation}
By repeating the fixed-point calculations for a range of values of the social effect parameter $q$, \pg{we compute the symmetric analogue} of the asymmetric pitchfork bifurcation reported in \citet{salhab2019collective}, for our setting of symmetric initial condition and destinations. We find that for $q<q_c=13.78$, the system has only the symmetric fixed point $v^*=0.5$, which is stable. At $q=q_c$, this fixed-point becomes marginally stable (i.e., $G'(0.5)\triangleq\dfrac{\partial G}{\partial v}|_{0.5}=1$), and for $q>q_c$ it is unstable. Two asymmetric, stable fixed-points $(v^*_+,v^*_-)$ are born at $q=q_c$, and they are related by $v^*_+=1-v^*_-$. The bifurcation diagram is shown in Fig. \ref{fig:bif}. Next, we establish some properties of the map $G$.

\begin{figure}[hbt!]
    \centering
    \includegraphics[width=0.99\linewidth]{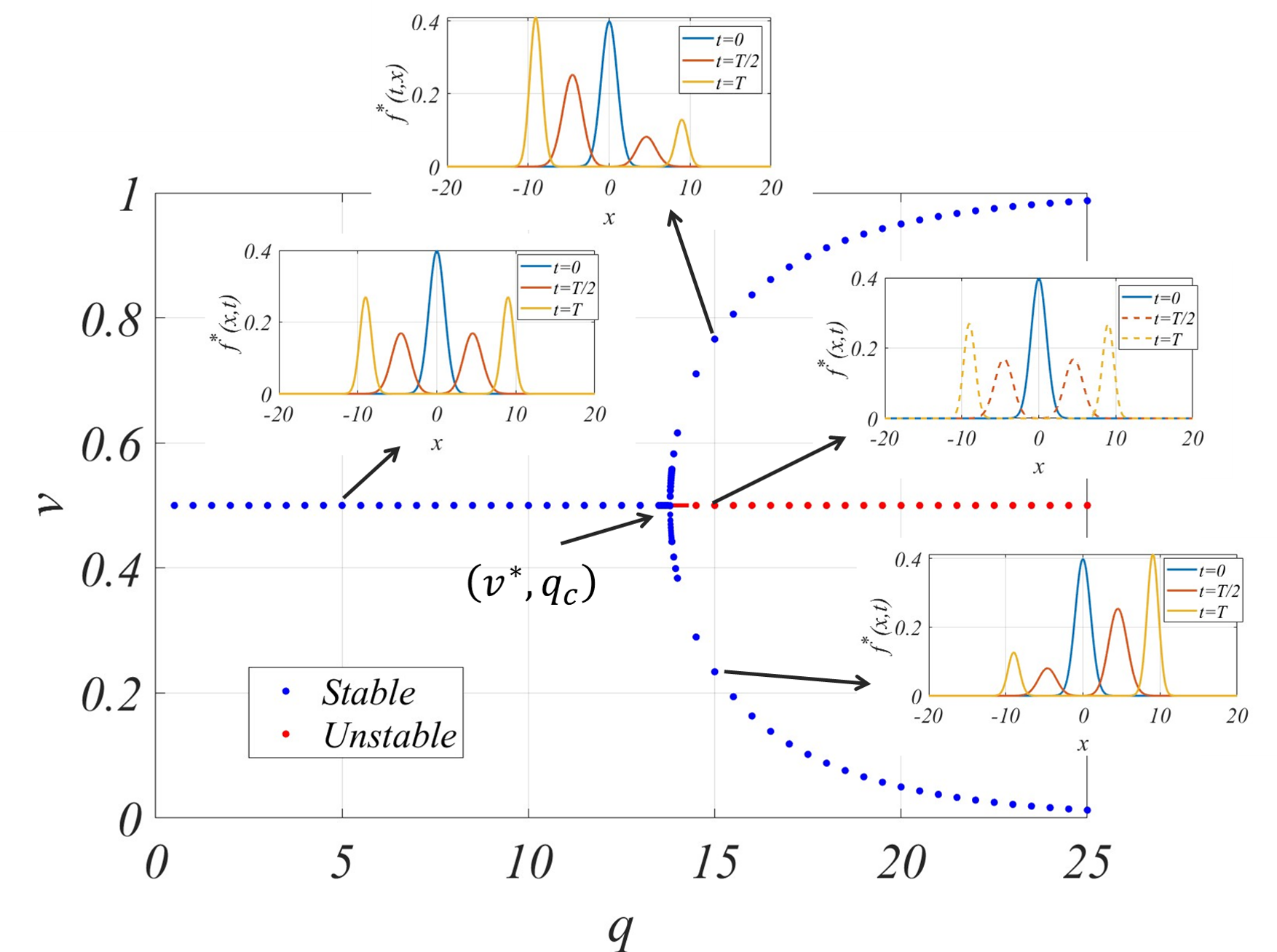}
    \caption{\footnotesize{Numerically computed bifurcation diagram for the map $G$. Each inset plot shows the initial, intermediate and final agent densities for the corresponding $v$.}}
    \label{fig:bif}
\end{figure}

\subsection{Odd Symmetry of $G$}
\label{sec:odd_G}
\begin{theorem} Suppose the initial density is the unit Gaussian with mean $X(0)=0$, and the final destination choices are symmetric about the origin, $y_2=-y_1$. Then, $G$ is odd symmetric about $v=0.5$, i.e., $G(1-v)=1-G(v)$. \end{theorem}

\begin{proof}
\textbf{Step 1: Symmetry of the mean $X(t;v)$}

With $y_1=-y_2$, and $X(0)=0$, Eq.\eqref{X mean} yields
\begin{equation}\label{X sym}
X(t;v)=R_2(t)\big(vy_1-(1-v)y_1\big)
      =R_2(t)(2v-1)y_1,
\end{equation}
and 
\begin{equation}
X(t;1-v)=R_2(t)\big((1-v)y_1+v y_2\big)
          =R_2(t)\big((1-v)y_1-vy_1\big)
          =R_2(t)(1-2v)y_1=-X(t;v). \label{eq:X_sym}
\end{equation}

\textbf{Step 2: Symmetry of the Riccati parameters \(s_j\) and \(d_j\)}

Recalling Eq.\eqref{hjb riccati2}, and using Eq. \ref{eq:X_sym}, we have
\begin{equation*}
\dot{s}_1(t;1-v)
=\left(\frac{b^2}{r}p(t)-a\right)s_1(t;1-v)-qX(t;v),
\end{equation*} 
and
\begin{equation*}
\dot{s}_2(t;v)
=\left(\frac{b^2}{r}p(t)-a\right)s_2(t;v)+qX(t;v).
\end{equation*}
The terminal conditions satisfy the relation $s_1(T;1-v)=-Cy_1=-s_2(T;v)$. This implies that $s_1(t;1-v)+s_2(t;v)$ solves a homogeneous linear ODE with zero terminal value, and hence, $s_1(t;1-v)=-s_2(t;v)$. Similar arguments can be employed to show that $d_1(t;1-v)=d_2(t;v)$.

\textbf{\pg{Step 3: Symmetry under action $v\rightarrow 1-v$ and $x\rightarrow -x$ with label exchange}}

For an arbitrary function $\mathcal{F}(t,x;v)$, we define the group action $\sigma \mathcal{F}(t,x;v):=\mathcal{F}(t,-x;1-v)$. The results of steps 1 and 2, in conjunction with definitions of Section \ref{subsec:map} can be used to conclude that $\sigma F_i=F_{3-i}$ for $V_i,g_i,\gamma_i$, while $\sigma F_i=-F_{3-i}$ for $u_i$ and $X_i$, for $i=1,2$.

\textbf{Step 4: Even symmetry of $f^*(t,x)$ under $\sigma$}

 We say that $\mathcal{F}$ is even under $\sigma$ if $\sigma\mathcal{F}=\mathcal{F}$, and odd if $\sigma\mathcal{F}=-\mathcal{F}$. The results in step 3 imply that $u^*$ is odd. A Fokker-Planck equation of the form given in Eq. \ref{FP f*_v} is equivariant w.r.t. $\sigma$ if the drift is \emph{odd} under $\sigma$. From step 3 and oddness of the linear term, the total drift $ax+bu^*$ is indeed odd. Since the initial condition is even under $\sigma$, i.e., $f^*(0,-x;1-v)=f^*(0,x;v)$, equivariance implies that $f^*(t,-x;1-v)=f^*(t,x;v)$ for all $t>0$. Using this relationship, and integrating over the positive reals,

\begin{equation*}
\int_0^\infty f^*(T,-x;1-v)\,dx
=
\int_0^\infty f^*(T,x;v)\,dx.
\end{equation*}
By a change of variable $x\rightarrow -x$ on the l.h.s, we obtain the desired result $G(1-v)=1-G(v)$.
\end{proof}
\subsection{Continuous dependence of $G$ and $G'$ on $q$}\label{sec:Gcont}
\pg{The continuity of $G$ in $q$ follows from continuity of the constituent maps. $G_4$ is independent of $q$, while the continuity of $G_3$ follows from standard ODE results on continuous dependence on parameters \citep{coddington1956theory}. 
As a result, the drift in the Fokker-Planck equation \ref{FP f*_v} depends continuously on $q$. The continuity of $G_1$ and $G_2$ then follows from standard results \citep{bogachev2016distances,bogachev2022fokker} on continuous parametric dependence of solutions of parabolic PDEs with drifts that have at most linear growth in $x$.

For every finite $Q$, the bounds in the arguments justifying chain rule in Sec. \ref{sec:linv*} and Appendix \ref{sec:a2} can be chosen uniformly for $q\in[0,Q]$, which yields continuity of $G'(v^*;q)$ in $q$ using continuity of $G$ and Taylor's theorem.}

\section{The Symmetric Fixed Point and the Linearized Map}
\label{sec:linearmap}
The goal of this section is to derive the linearized map $\hat{v}\mapsto G'(v^*)\hat{v}$ around the symmetric fixed point. We begin by analyzing the map at the fixed point.

\subsection{Properties of the Nash equilibrium corresponding to the symmetric fixed point}
Recall that for $v^*=0.5$, the mean \(X^*(t)=0\) from Eq.\eqref{X sym}. From results of Section \ref{sec:odd_G}, we obtain $s_1(t)=-s_2(t)$ where $s_1(t)=-y_1C\alpha(T,t)>0$ and $d_1(t)=d_2(t)$. This simplifies the expression for the weight in Eq. \ref{eq:gamma0} to
\begin{equation}\label{eq:gamma}
    \gamma_1(t,x)=\dfrac{g_1(t,x)}{g_1(t,x)+g_2(t,x)e^{2\eta s_1x}}.
\end{equation} We will also need its gradient

\begin{multline}\label{eq:gamma_x}
    \partial_x \gamma_1(t,x)=\\-2\gamma_1(t,x)\big(1-\gamma_1(t,x)\big)\bigg(\eta s_1(t)+\frac{\alpha(T,t)}{\pg{2}\sqrt{2\pi\Sigma(t)}}
\left[
\frac{\exp\left(-\dfrac{X_1(t,x)^2}{2\Sigma(t)}\right)}
{g_1(t,x)}
+
\frac{\exp\left(-\dfrac{X_2(t,x)^2}{2\Sigma(t)}\right)}
{g_2(t,x)}
\right]\bigg)\\\leq 0.
\end{multline}

The expression for the optimal control in Eq. \eqref{u*} reduces to
\begin{equation}\label{u^* sym}
u^*(t,x)=-\frac{b}{r}(p(t)x-s_1(t))-2\gamma_1(t,x) \frac{b}{r} s_1(t),
\end{equation}
with its gradient
\begin{equation*}
    \partial_xu^*(t,x)=-\frac{b}{r}p(t)-2\partial_x\gamma_1(t,x) \frac{b}{r} s_1(t).
\end{equation*}

\begin{remark}\label{rem:gamma1}
    While both $\partial_x\gamma_1$ and $\partial_xu^*$ are both uniformly bounded on each interval $[0,T_0]$ for $T_0<T$, the analysis of Eq. \ref{eq:gamma_x} reveals that they blow up $\sim O(\dfrac{1}{\sqrt{T-t}})$ at $x=0$ as $t\rightarrow T$.
\end{remark}

\subsection{Linearization around $v^*=0.5$}
\pg{\label{sec:linv*}}
Now we obtain a characterization of the linear map $\hat{v}\mapsto G'(v^*)\hat{v}$, evaluated at the symmetric solution $v^*=0.5$. From Eq. \ref{eq:maps}, 
\begin{equation*}
    G'(v^*)=DG_1|_{f^*}\underbrace{DG_2|_{u^*}\underbrace{DG_3|_{X^*}\underbrace{G'_4|_{v^*}}_{\hat{X}}}_{\hat{u}}}_{\hat{f}}.
\end{equation*}
We use $\hat .$  to denote linear responses of the maps $G_2$ through $G_4$ at the fixed point. Hence, using Eq. \ref{X sym}, we define
\begin{equation}\label{xh}
    \hat{X}(t;v^*)\triangleq G'_4|_{v^*}=2y_1R_2(t)=2y_1\dfrac{Cb^2e^{aT}}{2ar+Cb^2(e^{\pg{2a}T}-1)}(e^{at}-e^{-at}),
\end{equation} 

where we have used the explicit solution of $R_2(t)$ from Eq.\eqref{R2} with $R_2(0)=0$.

Using Eq. \ref{u*},
\begin{equation*}\label{uhat}
\hat u(t,x;v^*,q) \triangleq DG_3|_{X^*}\hat{X}= \hat{u}_2+\hat \gamma_1(u_1-u_2)+ \gamma_1(\hat u_1-\hat u_2),
\end{equation*}
where we have used $\gamma_2=1-\gamma_1$. Next,

\begin{equation*}\label{gamma hat}
\hat\gamma_1
=
\gamma_1\big(1-\gamma_1\big)
\left[
-\eta\big(\hat V_1-\hat V_2\big)
+\left(\frac{\hat g_1}{g_1}-\frac{\hat g_2}{g_2}\right)
\right],
\end{equation*}
where
\begin{equation}\label{Vhat uhat}
    \begin{split}
        &\hat V_j=\hat{s}_jx+\hat{d}_j,\qquad \hat u_j=-\frac{b}{r}\hat{s}_j,
    \end{split}
\end{equation}
and $\hat{s}_j $ and $\hat{d}_j$ solve the following linearized ODEs (see Eqs. \ref{hjb riccati2}),
\begin{equation}\label{shat dhat}
    \begin{split}
        & \dot{\hat{s}}_j=(\frac{b^2}{r}p-a)\hat{s}_j+q\hat{X} ,\qquad\dot{\hat{d}}_j= \frac{b^2}{r}s_j\hat{s}_j,
    \end{split}
\end{equation}
with the final conditions $\hat{s}_j(T)=0$ and $\hat{d}_j(T)=0$, respectively. This yields
\begin{equation}\label{sh}
   \begin{split}
        \hat{s}:=\hat{s}_1(t)=\hat{s}_2(t)=-\int_t^T q\hat{X}(\tau;v^*)\alpha(\tau,t)d\tau,
   \end{split}
\end{equation}
 Using $\alpha(T,\tau)>0$, $y_1<0$, and $a>0$, we conclude from Eqs. (\ref{xh},\ref{sh}) that $\hat X(t)\le0$, and $\hat{s}(t)\ge0$. From \eqref{sh}, $\hat{u}_1=\hat u_2$, and hence,
\begin{equation*}\label{uhat1}
\hat u
=
\hat u_1
+\eta\,\gamma_1\big(\gamma_1-1\big)\big(u_1-u_2\big)\big(\hat V_1-\hat V_2\big)
+\gamma_1\big(1-\gamma_1\big)\big(u_1-u_2\big)
\left(\frac{\hat g_1}{g_1}-\frac{\hat g_2}{g_2}\right),
\end{equation*}
where 
\begin{equation*}\label{u1u2}
u_1-u_2=-2\frac{b}{r}s_1(t).
\end{equation*}
Eq.\eqref{hjb riccati2} combined with final condition $s_1(T)=-y_1C$ yields
\begin{equation}
    \label{s_1(t)}
    s_1(t)=-y_1C\alpha(T,t)>0.
\end{equation}
From  Eqs. (\ref{Vhat uhat},\ref{shat dhat}),
\begin{equation*}
    \hat V_1(t)-\hat V_2(t)=\hat{d}_1(t)-\hat{d}_2(t)=\int_t^T 2y_1C\alpha(T,\tau)\frac{b^2}{r}\hat{s}(\tau)d\tau\le0.
\end{equation*} 
Next, from Eq.\eqref{eq:gj}, we get
\begin{equation*}
\hat g_j
=\frac{1}{\sqrt{2\pi \Sigma(t)}} \
\int_{W_j}
\hat X_j \frac{y - X_j}{\Sigma(t)}
\exp\!\left(-\frac{(y - X_j)^2}{2\Sigma(t)}\right)
  \, dy.
\end{equation*}
where \( \hat X_g(t)\triangleq \hat X_j(t) =-\int_t^T \alpha(T,\tau)\frac{b^2}{r}\hat s(\tau)\,d\tau\le0\) from Eq.\eqref{Xj Sigma}. It follows that

\begin{equation*} \label{g1hat g2hat}
    \hat g_1=-\frac{\hat X_g(t)}{\sqrt{2\pi\Sigma(t)}}
\exp\left(-\frac{X_1^2}{2\Sigma(t)}\right)\;\text{and}\; \hat g_2=\frac{\hat X_g(t)}{\sqrt{2\pi\Sigma(t)}}
\exp\left(-\frac{X_2^2}{2\Sigma(t)}\right).
\end{equation*} 

Hence,
\begin{multline}\label{uhat2}
\hat u(t,x;v^*,q)
=
-\frac{b}{r}\hat s(t)
+2\eta\frac{b}{r}s_1(t)\gamma_1(t,x)\big(1-\gamma_1(t,x)\big)\big(\hat V_1(t)-\hat V_2(t)\big)\\
+2\frac{b}{r}s_1(t)\gamma_1(t,x)\big(1-\gamma_1(t,x)\big)
\frac{\hat X_g(t)}{\sqrt{2\pi\Sigma(t)}}
\left[
\frac{\exp\left(-\dfrac{X_1(t,x)^2}{2\Sigma(t)}\right)}
{g_1(t,x)}
+
\frac{\exp\left(-\dfrac{X_2(t,x)^2}{2\Sigma(t)}\right)}
{g_2(t,x)}
\right].
\end{multline}
So, $\hat{u}(t,x;v^*,q)$ is the sum of three negative terms.
\begin{remark}\label{rem:gammahat1}
    From Eq. \ref{uhat2}, one can show that $\partial_x\hat{u}$ stays bounded as $t\rightarrow T$ at $x=0$, unlike $\partial_xu^*$. Hence, $\partial_x\hat{u}$ is uniformly bounded on $[0,T]\times \mathbb{R}$.
\end{remark}
The proof of the following Lemma is provided in the Appendix \ref{sec:a2}. 
\pg{\begin{lemma} Let $U_{T}$ denote the space of controls $u(t,x)$
$$
U_{T}
=
\left\{u:[0,T]\times \mathbb{R} \rightarrow \mathbb{R}:
u\in C^{\infty}([0,T)\times \mathbb{R}), u(T,.)=0,
\sup_{(t,x)} \frac{|u(t,x)|}{1+|x|}<\infty
\right\},
$$
equipped with the norm
$
\|u\|_{U_{T}}
=
\sup_{(t,x)} \dfrac{|u(t,x)|}{1+|x|}.
$ Then the map $G_2:U_{T}\mapsto C([0,T],L^1)$ is Fréchet differentiable at $u^*(t,x;v^*,q)$, and the derivative $\hat{f}^w(t,x;v^*,q) \triangleq DG_2|_{u^*}w$ satisfies the following PDE,
\begin{equation}\label{FP fhat}
    \partial_t\hat{f}^w=\frac{\sigma^2}{2}\,\partial_{xx}\hat{f}^w
-\partial_x\Bigl((ax+bu^*(t,x;v^*,q))\,\hat{f}^w +bw\,f^*(t,x;v^*,q)\Bigr), 
\end{equation}
for every $w\in U_T$, and $0< t< T$ with $\hat f^w(0,.)=0.$  \label{lem:fhat}
\end{lemma}

Applying the above lemma to $w=\hat{u}$, \begin{equation}
   G'(v^*;q)=DG_1|_{f^*}\hat{f}=\int_{x\le0}\hat{f}(T,x;v^*,q)\,dx \label{eq:dgdv},
\end{equation} where we use the simplified notation $\hat{f}=\hat{f}^{\hat{u}}$.}

\pg{In the above equation we have used the chain rule which can be justified as follows. For fixed $q$, differentiation of the explicit expressions in Eqs. \ref{u*}–\ref{Xj Sigma}, together with standard Gaussian derivative bounds, gives a locally uniform bound on $\partial_v^2u^*$ in $U_T$, and hence differentiability of $v\rightarrow (G_3\circ G_4) (v)$ at $v^*$ with derivative $\hat{u}$.}

\section{Stability analysis of the symmetric equilibrium, and Emergence of asymmetric equilibria}\label{sec:stab}
In this section, we analyze the linear stability of the symmetric fixed point of map $G$ as the social effect $q$ is increased from zero. It is easy to verify from the results of Sec. \ref{sec:linearmap} that $G'(v^*=0.5;q=0)=0$. From continuity (see Sec. \ref{sec:Gcont}), $|G'(v^*=0.5;q)|<1$ for small enough $q$, and hence the symmetric fixed point $v^*=0.5$ is stable when the social effect is small enough. Now we are ready to state the main theorem of the paper.
\begin{theorem}\label{th_main}
    \pg{There exists $q_0>0$ such that the symmetric fixed point $v^*=0.5$ is an unstable fixed point of the map $G$ for $q>q_0$, i.e., $G'(v^*;q)>1$ for all $q>q_0$.}
\end{theorem}

The key idea enabling the proof of this theorem is derivation of an auxiliary ODE, as described next.

\subsection{A non-negative integral representation for $G'$} In this section, we derive an auxiliary ODE (in time) whose solution has a non-negative integral formulation and equals $G'(v^*;q)$ at $t=T$. Let $\mathcal{L}_{v^*}$ be the Fokker-Planck operator corresponding to optimal control $u^*(t,x;v^*,q)$, i.e.,
\begin{equation}\label{eq:L*}
\mathcal{L}_{v^*}\triangleq\frac{\sigma^2}{2}\partial_{xx}(\cdot)-\partial_x\!\big( (a x+b u^*(t,x;v^*,q))(\cdot)\big),
\end{equation}
and define $H_{v^*}(t,x):=a x+b u^*(t,x;v^*,q)$. 

Let $\Psi(t,x;v^*,q)$ be the weak solution in backward time of the following adjoint PDE 
\begin{align}
\partial_t\Psi=-\mathcal{L}^A_{v^*}(\Psi),\label{eq:Psidot}
\end{align}
with final condition $\Psi(T,x)=\mathbf{1}_{x\leq 0}$. Here, $\mathcal{L}^A_{v^*}=\dfrac{\sigma^2}{2}\partial_{xx}(.)+H_{v^*}(t,x)\partial_x(.)$ is the $L^2-$adjoint of the forward operator $\mathcal{L}_{v^*}$.

\begin{remark}
The existence of a unique, weak, bounded solution $\Psi\in L^{\infty}([0,T]\times \mathbb{R})$ of Eq. \ref{eq:Psidot} can be shown as follows.  We mollify the non-smooth terminal data \citep{evans2022partial}, and construct the corresponding classical bounded, smooth solution $\Psi_{\epsilon}$. Since the drift is at most linear in $x$, maximum principle yields that $0\leq \Psi_{\epsilon}\leq 1$. Using standard interior regularity results \citep{friedman2008partial} and passing to the limit $\epsilon\rightarrow 0$ yields the weak solution with the given non-smooth terminal data. Uniqueness follows from comparison principle of bounded solutions \citep{friedman2008partial}.\label{rem:Psi}
\end{remark}

\begin{lemma}\label{lemma:gamma1}
The weight $\gamma_1(x,t)$ is a solution of the adjoint PDE Eq. \ref{eq:Psidot}, and therefore, $\Psi(t,x)=\gamma_1(t,x).$
\end{lemma}
\begin{proof}
 The statement can be verified by substituting the explicit expression of $\gamma_1$ in Eq. \ref{eq:gamma} into the adjoint PDE, where $u^*$ is given by Eq. \ref{u^* sym}. Note that $\gamma_1$ appears in the expression of $u^*$, so Eq. \ref{eq:Psidot} is in fact a \emph{nonlinear} PDE in $\gamma_1$.
\end{proof}

Next, we define a new function $Z(t;v^*,q)\triangleq  \bigintssss_{\mathbb{R}}\hat f(t,x;v^*,q)\Psi(t,x;v^*,q) dx$.
From Eq. \ref{eq:dgdv}, and the final time condition on $\Psi$, it follows that $G'(v^*;q)=Z(T;v^*,q)$. 

\begin{lemma}
$Z$ satisfies the ODE $\dot{Z}=\bigintssss_{\mathbb{R}}b\hat{u}(t,x)f^*(t,x)\partial_x\Psi(t,x)dx$, with initial condition $Z(0)=0$.

\end{lemma}

\begin{proof}
Direct differentiation yields
\begin{equation*}\label{zdot2}
    \dot{Z}=\int_\mathbb{R} \big((\partial_t\hat{f})\Psi  +\hat{f}\partial_t\Psi\big)\,dx.
\end{equation*}
From Eq. \ref{FP fhat}, $\partial_t\hat{f}=\mathcal{L}\hat{f}-\partial_x(b\hat{u}f^*)$. Hence,
\begin{align*}
\dot{Z}=\int_\mathbb{R} \big(\mathcal{L}[\hat{f}]\Psi  +\hat{f}\partial_t\Psi\big)\,dx-\int_\mathbb{R} \Psi(t,x)\partial_x(b \hat{u}(t,x)f^*(t,x))dx.
\end{align*}
Using Eq. \ref{eq:Psidot} in the first integral, and applying integration by parts on the second integral, 
\begin{align*}
\dot{Z}=\int_\mathbb{R} \big(\mathcal{L}[\hat{f}]\Psi  -\hat{f}\mathcal{L}_A\Psi)\,dx+\int_\mathbb{R} b \hat{u}(t,x)f^*(t,x)\partial_x(\Psi(t,x))dx.
\end{align*}
But the first integral is $0$ by using the property of adjoints, hence we obtain the desired ODE. The initial condition $Z(0)=0$ follows from the fact that $\hat{f}(0,x)=0$.
\pg{
Since $f^*$ has uniformly bounded first moment, and $\hat{f}$ is uniformly bounded in $L^1$ on each $[0,T_0]$ as shown in the Appendix \ref{sec:a2}, standard spatial cut-off arguments justify the integrations by parts above on $[0,T_0]$ for all $T_0<T$. 
}
\end{proof}

Next, we define $h(t,x)\triangleq -\partial_x\Psi(t,x)$. Differentiating Eq. \ref{eq:Psidot} w.r.t. $x$, we obtain a Fokker-Planck equation for the evolution of $h$,
\begin{align} \label{eq:h_FP}
    \partial_th(t,x)=-\dfrac{\sigma^2}{2}\partial_{xx}h(t,x)-\partial_x(H_{v^*}(t,x) h(t,x)).
\end{align}
with final condition $h(T,x)=\delta_0(x)$. Since  $h$ satisfies a Fokker-Planck equation with positive final condition, $h(t,x)\ge0$ for all $0\leq t\leq T$. 

Finally, we integrate the $Z$ ODE to obtain the integral form of $Z$,
\begin{equation}\label{Z(T;q)}
    Z(T;q)=\int_0^T\int_\mathbb{R} b (-\hat{u}(t,x;v^*,q))f^*(t,x;v^*,q)h(t,x;v^*,q)\,dxdt.
\end{equation}
A key property of Eq. \ref{Z(T;q)} is that each term in the integrand is nonnegative. \pg{Recall that Lemma \ref{lem:fhat} provides $\hat{f}(t)\rightarrow\hat{f}(T)$. Together with $0\leq\Psi\leq 1$, and $\Psi(t,x)\rightarrow{\mathbf{1}}_{x<0}$, this shows that $Z(t)\rightarrow G'(v^*;q)$ by dominated convergence. Since the integrand is nonnegative, we obtain Eq. \ref{Z(T;q)} as $t\rightarrow T$.}

The proof of Theorem \ref{th_main} will be obtained by constructing a monotonic in $q$ lower bound on $G'(v^*;q)=Z(T;v^*,q)$, exploiting the integral formulation of Eq. \ref{Z(T;q)}. We will lower bound the integral by its lower bound on the subset $([T - \delta T(q), T] \times [-\delta x(q), \delta x(q)])$ of the full domain $[0,T]\times\mathbb{R}$, where the size of the subset decreases monotonically in $q$. Hence we will use,
\begin{equation}\label{lb_z_subset} 
Z(T;q)\ge\int_{T-\delta T(q)}^T\int_{-\delta x(q)}^{\delta x(q)} b (-\hat{u}(t,x))f^*(t,x)h(t,x)\,dx\,dt.
\end{equation}

This strategy allows the use of local estimates near the point $(t=T,x=0)$. This is because \( h(T,x) \) and \( \hat{u}(T,x) \) are explicitly known, and the lower bounds of \( f^*(t,x) \) and \( h(t,x) \) will involve estimating Gaussians centered at $x=0$. We obtain the lower bounds of the three terms in the integrand of Eq. \ref{lb_z_subset} in the next three subsections, and combine those bounds to provide the proof of Theorem \ref{th_main} in the final subsection.

\begin{remark}
Here, $h$ is understood as a measure-valued solution of Eq. \ref{eq:h_FP}, and a distributional derivative of $\Psi$. Its existence can be shown using a mollification process as in Remark \ref{rem:Psi}. Classical solutions of the (time-reversed) problem exist for the problem with mollified initial data. Their positivity follows from the parabolic maximum principle. Since the Fokker-Planck Cauchy problem is well-posed for measure-valued initial data, we can pass to the limit to obtain the desired measure-valued solution $h$ \citep{bogachev2022fokker}.
\end{remark}

In the estimates below, the generic constants appearing as prefactors are all positive numbers independent of $q$; their values may change from line to line. Any $q$ dependence of other constants will be written explicitly.

\subsection{Lower bound of $-\hat{u}(t,x)$} \label{sec:lb_uhat}

Recall from Eq.\eqref{uhat2} that $-\hat u$ is a sum of three positive terms. Keeping only the first term in the lower bound and using the expression for $\hat{s}(t)$ from Eq. \ref{sh},

\begin{equation}\label{uh lb}
    -\hat{u}(t,x)\ge\dfrac{b}{r}\hat{s}(t)=C_1\int_t^T q\alpha(\theta,t)(e^{a\theta}-e^{-a\theta})\,d\theta, 
\end{equation}
with $C_1=2\dfrac{b}{r}|y_1|\dfrac{Cb^2e^{aT}}{2ar-Cb^2+Cb^2e^{\pg{2a}T}}$. 

From Eqs. (\ref{alpha ineq}, \ref{alphaTt ineq}), we recall the bounds
\begin{equation*}\label{alpha bound1}
    \dfrac{1-K(q)}{e^{\Delta(q)(T-t)}-K(q)e^{-\Delta(q)(T-t)}}\le\alpha(\theta,t)\le1.
\end{equation*}

While the above formula has a complicated dependence on $q$, we can get $q-$independent bound on $\alpha(\theta,t)$ for large enough $q$, in an appropriately chosen time window $[T-\delta T(q),T]$. Hence we make the following choice:
\begin{align}\label{eq:deltaT}
   \delta T(q)\triangleq  \dfrac{1}{\Delta(q)},
\end{align} where recall that $\Delta(q) = \sqrt{a^2 + b^2q/r}$. It follows that for $\theta$ and $t$ in the window, $T-t\leq \dfrac{1}{\Delta(q)}$, and
\begin{equation*}\label{alpha bound}
     \alpha(\theta,t)\geq  \dfrac{1-K(q)}{e^{\Delta(q)(T-t)}-K(q)e^{-\Delta(q)(T-t)}}\geq \dfrac{1-K(q)}{e-K(q)}\geq \dfrac{1}{e+1}\triangleq C_2,
\end{equation*}
where we have used that 
\begin{align}\label{eq:Kbound}
    -1\le K(q)\le0
\end{align} for large enough $q$, from Eq. \ref{eq:K(q)}.

Let $F_u(\theta)=e^{a\theta}-e^{-a\theta}$. By continuity of $F_u(\theta)$, we can find a small enough time window $\theta\in [T-\delta T(q),T]$ s.t. $F_u(\theta)>\dfrac{F_u(T)}{2}\triangleq C_3$, independent of $q$. Since window size is monotonically decreasing in $q$, this is achievable by picking large enough $q$.
 
Inserting the above two bounds in Eq. \ref{uh lb}, we get the final lower bound on $\hat{u}(t,x)$,
\begin{equation}
   \label{lb uhat2}
    -\hat{u}(t,x)\ge C_uq(T-t),
\end{equation}
which holds for large enough $q$, in the time window $t\in [T-\dfrac{1}{\Delta(q)},T]$.

\subsection{Lower bound of $f^*$} 
\pg{\label{sec:lbf*}}
We rewrite the Fokker-Planck equation for $f^*(t,x)$ by splitting its drift as
\begin{equation}\label{FP o-u process}
    \partial_tf^*=\frac{\sigma^2}{2}\partial_{xx}f^*-\partial_x\big(A(t)xf^*\big)-\partial_x\big(B(t,x) f^*\big),
\end{equation}
where, $A(t)=a-\dfrac{b^2}{r}p(t)$, and $B(t,x)\triangleq(1-2\gamma_1(t,x)) \dfrac{b^2}{r} s_1(t)$, using Eq.\eqref{u^* sym}. Our strategy here will be to employ the Duhamel formula, treating the third term on the r.h.s in Eq. \ref{FP o-u process} as an inhomogeneous `source' term. The above split is useful here because of the following properties: 
\begin{itemize}
\item{Uniform upper bound on $B(t,x)$ on $[0,T]\times \mathbb{R}$}:
Since the weight satisfies $0\leq \gamma_1\leq 1$, and the $\|s_1(t)\|_{L^\infty([0,T])} \le|y_1|C$ from Eq.\eqref{s_1(t)}, we get
\begin{equation}
    \label{ub B}\|B(t,x)\|_{L^\infty(\mathbb{R}\times[0,T])}\le\dfrac{\pg{b^2}}{r}\|1-2\gamma_1(t,x)\|_{L^\infty(\mathbb{R}\times[0,T])}\|s_1(t)\|_{L^\infty([0,T])}\le\dfrac{\pg{b^2}}{r}|y_1|C:=l_B.
\end{equation}

\item{Non-negativity of $\partial_x B(t,x)$}:

The derivative
\begin{align}\label{eq:d_xB}
    B_x(x,t)=-2\partial_x\gamma_1(x,t)\dfrac{\pg{b^2}s_1(t)}{r}\geq 0,
\end{align} 
using Eqs. (\ref{eq:gamma_x},\ref{s_1(t)}).
\end{itemize}

Applying Duhamel's formula to Eq. \ref{FP o-u process} \pg{with $t_0=0$},
\begin{equation}\label{fv sol}
    \begin{split} 
        &f^*(t,x)=\int_\mathbb{R}F_h(t,x,t_0,\xi)f^*(t_0,\xi)d\xi-\int_0^t\int_\mathbb{R}F_h(t,x,s,\xi)\partial_\xi(B(s,\xi)f^*(s,\xi))d\xi ds\\& =\underbrace{\int_\mathbb{R}F_h(t,x,t_0,\xi)f^*(t_0,\xi)d\xi}_{f_h(t,x)}+\underbrace{\int_0^t\int_\mathbb{R}\partial_\xi F_h(t,x,s,\xi)B(s,\xi)f^*(s,\xi)d\xi ds}_{f_n(t,x)},
    \end{split}
\end{equation}
where $f_h$ is the homogeneous solution, and $f_n$ is the non-homogeneous part. In Eq. \ref{fv sol}, $F_h$ is the fundamental solution of Eq. \ref{FP o-u process} with drift $A(t)x$,
\begin{equation}\label{F_h}
F_h(t,x,s,\xi)=\dfrac{1}{\sqrt{2\pi \Sigma_h(t,s)}}\exp\left(-\dfrac{(x-\Phi_h(t,s)\xi)^2}{2\Sigma_h(t,s)}\right),
\end{equation}
with $\Phi_h(t,s)=e^{\bigintssss_s^tA(\tau)d\tau}$ and $\Sigma_h(t,s)=\sigma^2\bigintssss_s^te^{2\bigintssss_\tau^tA(u)du}d\tau$.
The derivative of $F_h$  with respect to \(\xi\) is 
\begin{equation}\label{DF_h}
\partial_\xi F_h=\frac{\Phi_h(t,s)(x-\Phi_h(t,s)\xi)}{\Sigma_h(t,s)}F_h.
\end{equation}

\subsubsection{Lower bound of \texorpdfstring{$f_h(t,x)$}{f_h(t,x)}}\label{sec:lb_fh}
With the symmetric initial condition $f^*(0,x)=\mathcal{N}(0,1)$, the homogeneous solution is the Gaussian
\begin{align}
  f_h(t,x)=\dfrac{1}{\sqrt{2\pi\Sigma(t)}}\exp\left(-\dfrac{x^2}{2\Sigma(t)}\right),\label{eq_fh}
\end{align}
with $\dot\Sigma(t)=2(a-\dfrac{b^2p(t)}{r})\Sigma(t)+\sigma^2,$ and $\Sigma(0)=1$.

Using the expression for $p(t)$ from Eq. \ref{alpha p(t)}, we obtain an explicit expression for $\Sigma(t)$,
\begin{equation*}
    \Sigma(t)=e^{-2\Delta(q) t}\pg{D(q,t)^2}+\sigma^2 \pg{D(q,t)}\dfrac{1-e^{-2\Delta(q) t}}{2\Delta(q)},
\end{equation*}
where $\pg{D(q,t)}\triangleq \dfrac{1-K(q)e^{-2\Delta(q) (T-t)}}{1-K(q)e^{-2\Delta(q) T}}$.

Again using Eq. \ref{eq:Kbound},
\begin{equation*}
    1\le \pg{D(q,t)}\le \dfrac{1-K(q)}{1-K(q)e^{-2\Delta(q) T}}\le1-K(q)\le2,
\end{equation*}
and thus $\Sigma(t)$ satisfies
\begin{align*}
    \sigma^2 \dfrac{1-e^{-2\Delta(q) t}}{2\Delta(q)} \le \Sigma(t)\le 4e^{-2\Delta(q) t}+ \dfrac{\sigma^2}{\Delta(q)}.
\end{align*}

For $t\in[T-\delta T(q),T]$ with $\delta T=\dfrac{1}{\Delta(q)}$ as before, we have $e^{-2\Delta(q) T}\le e^{-2\Delta(q) t}\le e^{-2\Delta(q) T+2}$. Exploiting this, we define $\epsilon_{u}(q)\triangleq \Delta( q)4e^{-2\Delta( q) T+2}+\sigma^2$, and $\epsilon_{l}(q)\triangleq 1-e^{-2\Delta( q) T+2}$. This implies that $\pg{\dfrac{\sigma^2}{2}}\epsilon_l(q)\leq \Delta(q)\Sigma(t)\leq \epsilon_u(q)$, which when inserted in Eq. \ref{eq_fh} yields the following lower bound on $f_h$,
\begin{align*}
  f_h(t,x)\ge\dfrac{\sqrt{\Delta(q)}}{\sqrt{2\pi\epsilon_u(q)}}\exp\left(\frac{-\Delta(q)x^2}{\sigma^2\epsilon_l(q)}\right).
\end{align*}

Since $\epsilon_u$ is an eventually decreasing function of $q$, and $\epsilon_l$ is a monotonically increasing function of $q$, we can choose a $q_f$ such that for all $q>q_f$, $\epsilon_u(q)<\epsilon_u(q_f)$, and $\epsilon_l(q)>\epsilon_l(q_f)$. This implies that for $q>q_f$, 
\begin{equation*}
    f_h(t,x)\ge  \dfrac{\sqrt{\Delta(q)}}{\sqrt{2\pi\epsilon_{u}(q_f)}}\exp\left(-\dfrac{\Delta(q)x^2}{\sigma^2 \epsilon_{l}(q_f)}\right).
\end{equation*}

To eliminate the $q$ dependence of the exponential, we make the following choice of the spatial window around $x=0$: 
\begin{align}\label{eq:deltax}
\delta x(q)=\sqrt{\delta T(q)}=\dfrac{1}{\sqrt{\Delta(q)}}.\end{align} 
Hence, in the space-time window $[-\delta x(q),\delta x(q)]\times [T-\delta T(q),T]$, and for large enough $q$, the above bound reduces to
\begin{equation}\label{lb f_h}
    f_h(t,x)\ge C_f\sqrt{\Delta(q)}.
\end{equation}

\subsubsection{Lower bound of \texorpdfstring{$f_n(t,x)$}{f_n(t,x)}}
We lower bound the nonhomogeneous term by using the Cauchy-Schwarz inequality,
\begin{equation}
\left |\int_\mathbb{R}\partial_\xi F_h(t,x,s,\xi)B(s,\xi)f^*(s,\xi)d\xi\right |\leq \|\partial_\xi F_h(t,x,s,.)\|_{L^2_{\xi}}\|B(s,.)f^*(s,.)\|_{L^2_{\xi}}, \label{fn_CS}
\end{equation}
which yields the lower bound $f_n(t,x)\geq -\int_{0}^t l_b \|\partial_\xi F_h(t,x,s,.)\|_{L^2_{\xi}}\|f^*(s,.)\|_{L^2_{\xi}}ds$, using the upper bound on $B$ from Eq. \ref{ub B}. 

\paragraph{\textbf{An $L^2$ upper bound on $f^*$:}}
Multiplying Eq. \eqref{FP o-u process} with $f^*$ on both sides and integrating over the spatial coordinate, 
\begin{equation*}\begin{split}
\partial_t \|f^*\|^2_{L_x^2}=-\sigma^2\|\partial_xf^*\|^2_{L_x^2}-\pg{2}\int_{\mathbb{R}}f^*(x,t)\partial_x((A(t)x+B(x,t))f^*(x,t))dx\\
=-\sigma^2 \|\partial_xf^*\|^2_{L_x^2}-\int_{\mathbb{R}}(f^*(x,t))^2\partial_x(A(t)x+B(x,t))dx,\label{fstar_L2}
\end{split}
\end{equation*}
where we have repeatedly used integration by parts. 

Next, we use $A(t)\geq -\Delta(q)$ (proved in derivation of Eq. \ref{expbnd_phi} below) and $\partial_xB(x,t)>0$ (from Eq. \ref{eq:d_xB}) in the above equation, to get
\begin{align*}\partial_t \|f^*\|^2_{L_x^2}\leq -\sigma^2\|\partial_xf^*\|^2_{L_x^2}+\Delta(q)\|f^*\|^2_{L_x^2}. \end{align*}

To obtain a closed inequality from above, we use the following 1D Nash inequality \citep{nash1958continuity}:
\begin{align*} \|f^*\|^6_{L_x^2}\leq C_1 \|f^*\|^4_{L_x^1}\|\partial_xf^*\|^2_{L_x^2},\end{align*} and $\|f^*\|_{L_x^1}=1$ to get the differential inequality
\begin{align*}\dot{L}(t) \leq \Delta(q)L(t)-(\sigma^2/C_1)L(t)^3,
\end{align*} where $L\triangleq\|f^*\|^2_{L_x^2}$. It is straightforward to conclude that for positive initial conditions, $L(t)\leq \max\{L(0),\sqrt{\Delta(q)C_1/\sigma^2}\}$. 
This gives us our desired $s-$independent bound for large enough $q$:
\begin{align} \|f^*(s,.)\|_{L_x^2}\leq C_1 \Delta(q)^{1/4}. \label{fv_l2} \end{align} 
\paragraph{\textbf{An upper bound on $\int_0^t \|\partial_{\xi}F_h\|_{L^2_{\xi}}ds$:}}

From Eq. \ref{DF_h}, we obtain exact expression \begin{align*} \|\partial_\xi F_h(t,x,s,.)\|_{L^2_{\xi}}=\dfrac{\sqrt{\Phi_h(t,s)}}{2\pi^{1/4}}\Sigma_h(t,s)^{-3/4}, \end{align*} which we bound for $t\in (T-\dfrac{1}{\Delta(q)},T)$, the previously chosen final window. 
We set $\hat{t}=t-s$ and evaluate the above integral by splitting the time domain into a near and a far integral: $\hat{t} \in (0,\hat{t}_0\triangleq1/\Delta(q))$, and $\hat{t} \in [\hat{t}_0,t).$ We will need the following additional bounds:
\begin{itemize}
\item Uniform lower bound on $\Sigma_h$ for all $\hat{t}$: For $\Sigma_h$, we use that $A(.)\geq -\Delta(q)$ for $q$ large enough, and hence, $\bigintssss_{\tau}^t Ad\tau\geq -\Delta(q)(t-\tau)$. From Eq. \ref{F_h}, $\Sigma_h(t,t-\hat{t})=\sigma^2\bigintsss_{t-\hat{t}}^t \exp(2\bigintsss_{\tau}^t A(t_1)dt_1)d\tau \geq \sigma^2 \bigintsss_{t-\hat{t}}^t \exp(-2\Delta(q)(t-\tau)) d\tau$. Hence, the required lower bound is:
\pg{\begin{align} \Sigma_h(t,t-\hat{t})\geq\dfrac{\sigma^2}{2\Delta(q)}(1-e^{-2\Delta(q)\hat{t}})\label{lb_sig}\end{align}  }

\item Exponentially decaying upper bound on $\Phi_h$ for $\hat{t}>\dfrac{1}{\Delta(q)}$:  Recall from Eq. \ref{def A(t)}
: $A(\tau)=-\Delta(q)\dfrac{1+K(q)e^{-2\Delta(q)(T-\tau})}{1-K(q)e^{-2\Delta(q)(T-\tau)}}<0$, where for $q$ large enough, $-1\leq K(q) \leq 0$ from Eq. \ref{eq:Kbound}. Now for $\tau \leq T-\dfrac{1}{\Delta(q)}$, $e^{-2\Delta(q)(T-\tau)}\leq e^{-2}$. Hence, $A(\tau)=-\Delta(q)\dfrac{1-\hat{K}(q)}{1+\hat{K}(q)}$, where $\hat{K}(q)\triangleq -K(q)e^{-2\Delta(q)(T-\tau)} \in [0,e^{-2}]$. This implies $\dfrac{1-\hat{K}(q)}{1+\hat{K}(q)} \pg{> \dfrac{1-e^{-2}}{1+e^{-2}}}:=C_2<1$, with $C_2$ independent of $q$, and $A(\tau)\leq -C_2\Delta(q)$. 

We get a uniform exponential bound as follows:
\begin{equation}
\label{expbnd_phi}
\begin{aligned}
\Phi_h(t,t-\hat{t})
&= \exp\left\{
\begin{aligned}
&\int_{t-\hat{t}}^{T-1/\Delta(q)}
\underbrace{A(\tau)}_{\leq -C_2\Delta(q)}\,d\tau 
&\quad
+\int_{T-1/\Delta(q)}^{t}
\underbrace{A(\tau)}_{\leq 0}\,d\tau
\end{aligned}
\right\} \\
&\leq
\exp\left\{
\int_{t-\hat{t}}^{T-1/\Delta(q)}
-C_2\Delta(q)\,d\tau
\right\}
\leq
\exp\!\left[
-C_2\Delta(q)
\left(\hat{t}+T-t-\frac{1}{\Delta(q)}\right)
\right] \\
&\leq
\exp\!\left[
-C_2\Delta(q)
\left(\hat{t}-\frac{1}{\Delta(q)}\right)
\right] \\
&= C_3 \exp\!\left(-C_2\Delta(q)\hat{t}\right).
\end{aligned}
\end{equation}
\end{itemize}

We are now ready to evaluate the near and far integrals.
\begin{itemize}
\item Integral over $0\leq \hat{t} \leq 1/\Delta(q)$ : In this interval, we use the inequality $\dfrac{1-e^{-m}}{m}\geq \dfrac{1-e^{-2}}{2}$ for $0< m\leq 2$. By putting $m=2\Delta(q)\hat{t}$,  Eq. \ref{lb_sig} yields $\Sigma_h(t,t-\hat{t})\geq \dfrac{\sigma^2}{2}(1-e^{-2})\hat{t}$. In addition, we use the rough bound $\sqrt{\Phi_h(t,t-\hat{t})}\leq 1$. The integral is \begin{align*}
\int_{0}^{1/\Delta(q)} \sqrt{\Phi_h(t,t-\hat{t})}\Sigma_h(t,t-\hat{t})^{-3/4}d\hat{t} \leq C_4\int_{0}^{1/\Delta(q)} \hat{t}^{-3/4}d\hat{t}=C_4\Delta(q)^{-1/4}.\end{align*}

\item Integral over $\hat{t} \geq 1/\Delta(q)$ : For this interval, Eq. \ref{lb_sig} yields $\Sigma_h(t,t-\hat{t})\geq \dfrac{\sigma^2}{2\Delta(q)}(1-e^{-2})$, and hence $\Sigma_h(t,t-\hat{t})^{-3/4}\leq C_5\Delta(q)^{3/4}$. Using this bound along with Eq. \ref{expbnd_phi} to evaluate the integral, we get \begin{align*}
\int_{1/\Delta(q)}^{t} \sqrt{\Phi_h(t,t-\hat{t})}\Sigma_{\pg{h}}(t,t-\hat{t})^{-3/4}d\hat{t} \leq C_5\Delta(q)^{3/4}\int_{1/\Delta(q)}^{\infty} \exp(-C_2\Delta(q)\hat{t})d\hat{t}\\=C_5\Delta(q)^{-1/4}.
\end{align*}
\end{itemize}

Combining the two integrals above, we get
\begin{align}
\int_0^t \|\partial_{\xi}F_h\|_{L^2_{\xi}}ds \leq C_5\Delta(q)^{-1/4}.\label{bnd_Fh_deriv}
\end{align} 

Using the bounds given by Eqs. \ref{fv_l2} and \ref{bnd_Fh_deriv} in Eq. \ref{fn_CS}, we get the final lower bound on $f_n$
\begin{align}\label{eq:fnbnd}
f_n(t,x)\geq -C_1l_b\Delta(q)^{1/4} \Delta(q)^{-1/4}=-\epsilon_n,
\end{align}
where $\epsilon_n>0$ is independent of $q$.

\subsubsection{Final lower bound of $f^*(t,x)$}
Combining the lower bounds from Eqs.\eqref{lb f_h} and \eqref{eq:fnbnd} into Eq.\eqref{fv sol}, we get 
\begin{equation}\label{lb fv*2_1}
    f^*(t,x;v^*,q)\ge C_f\sqrt{\Delta(q)},
\end{equation}
for large enough $q$, in the chosen terminal space-time window.

\subsection{Lower bound of $h$}
One could obtain this bound using the explicit representation of \(h=-\partial_x\gamma_1\) from Eq.~\eqref{eq:gamma_x}. However, a more abstract argument is both cleaner and more direct, and we therefore follow that approach.
We will use the Aronson-Nash estimate \citep{aronson1967bounds,aronson1968non} to obtain a lower bound for $h$. However, this method cannot be applied directly to the governing equation Eq. \ref{eq:h_FP} due to the presence of the unbounded linear drift term. Since we want to exploit the known final time condition on $h$ for obtaining the lower bounds, we first obtain the time-reversed PDE
\begin{equation*}\label{h(tau_x)}
    \partial_\tau \tilde h(\tau,x)=\dfrac{\sigma^2}{2}\partial_{xx}\tilde h(\tau,x)+\partial_x(H(T-\tau,x) \tilde h(\tau,x)),
\end{equation*}
with the \emph{initial} condition $\tilde h(0,x)=\delta_0(x)$. Here, $\tilde h(\tau,x):=h(t,x)$.

To eliminate the unbounded linear drift term, we introduce a new spatial variable $y$, and a field $\rho_1$, via the relations $x=e^{\Lambda(\tau,0)}y$ and $\tilde h(\tau,x)=e^{-\Lambda(\tau,0)}\rho_1(\tau,y)$, where

\begin{align*}
    \label{e^Lam(tau)}\Lambda(\tau,0) =\pg{-}\int_{T-\tau}^{T}(a-b^2\,p(s)/r)ds.
\end{align*}
It follows from Eq. \ref{alpha p(t)} that $e^{\pg{-}{\Lambda(\tau,0)}}=\alpha(T,T-\tau)$, and hence from Eq. \ref{eq:h_FP}, we obtain the following PDE for $\rho_1(\tau,y)$,
\begin{equation}
    \label{FP trf_h}\partial_\tau\rho_1=\dfrac{\sigma^2}{2}e^{-2\Lambda(\tau,0)}\partial_{yy}\rho_1+e^{-\Lambda(\tau,0)}\hat B(T-\tau,y)\partial_y\rho_1+\hat B_{\pg{x}}(T-\tau,y)\rho_1,
\end{equation}
with the initial condition $\rho_1(0,y)=\delta_0(y)$. Here, $\hat B(T-\tau,y):=B(T-\tau,e^{\Lambda(\tau,0)}y)$, and $\hat B_{\pg{x}}(T-\tau,y):=B_x(T-\tau,e^{\Lambda(\tau,0)}y)$. Recall from Eqs. (\ref{ub B},\ref{eq:d_xB}) that first term is uniformly bounded on $[0,T]\times \mathbb{R}$, and the second term is positive. 

From Eq. \ref{alpha p(t)},
\begin{equation*}
    e^{-\Lambda(\tau,0)}=\pg{\dfrac{1-K(q)}{e^{\Delta(q) \tau}-K(q)e^{-\Delta(q) \tau}}}.
\end{equation*}

For our chosen terminal time window $\tau\in[0,\delta T(q)]$, and for $q$ large enough, arguments similar to those in Section \ref{sec:lb_uhat} yield that $C_1\leq e^{-\Lambda(\tau,0)} \leq C_2$. Hence, we conclude from the above discussion that in Eq. \ref{FP trf_h} in the terminal time window, the coefficients of the first and second order terms are bounded independently of $q$, and the coefficient of the zero order term is positive.

Now consider the following auxiliary PDE,

\begin{equation}
    \label{eq:frozen}\partial_{\pg{\tau}}\tilde \rho_1 =D_1\partial_{yy}\tilde\rho_1 +D_2 \partial_y\tilde\rho_1, 
\end{equation}
where 
\begin{equation*}
\begin{split}
        & \text{for}\, \tau\in [0,\delta T(q) ]
\left\{
\begin{array}{l}
D_1=\dfrac{\sigma^2}{2}e^{-2\Lambda(\tau,0) } \\ 
D_2=e^{-\Lambda(\tau,0) }\hat B(T-\tau,y) 
\end{array}
\right. \\& \text{for}\, \tau\in ( \delta T(q),T]
\left\{
\begin{array}{l}
D_1=\dfrac{\sigma^2}{2}e^{-2\Lambda(\delta T(q),0) } \\ 
D_2=e^{-\Lambda(\delta T(q),0) }\hat B(T-\delta T(q),y) .
\end{array}
\right.
\end{split}\end{equation*}
For $\tau\in [0,\delta T(q) ]$, Eq. \ref{eq:frozen} has the same coefficients for the first and second order terms as Eq. \ref{FP trf_h}, but its coefficients are `frozen' for $\tau>\delta T(q)$ to their value at $\tau=\delta T(q)$. By the above discussion, coefficients of Eq. \ref{eq:frozen} are bounded independently of $q$ for all $t$. Hence, the Aronson-Nash estimate for the fundamental solution of Eq. \ref{eq:frozen}, on $\tau\in[0,T]$, is
\begin{equation*}
    \dfrac{\tilde C_3}{\sqrt{\tau-s}}e^{-\dfrac{(y-\eta)^2}{ \tilde C_4(\tau-s)}}\le\tilde P_1(\tau,y,s,\eta),
\end{equation*}
where $\tilde C_3$ and $\tilde C_4$ only depend on $T$, and are independent of $q$. 

Now since the first order and second order coefficients of the two PDEs agree on the terminal time window, and the zero order coefficient in Eq. \ref{FP trf_h} is positive, $P_1(\tau,y,s,\eta)\geq\tilde{P}_1(\tau,y,s,\eta)$ over $\tau\in [0,\delta T(q) ]$ by the maximum principle, and hence,
\begin{equation*}\label{P1 sup}
     \dfrac{\tilde C_3}{\sqrt{\tau-s}}e^{-\dfrac{(y-\eta)^2}{ \tilde C_4(\tau-s)}}\le P_1(\tau,y,s,\eta).
\end{equation*}
\pg{The comparison for the delta initial datum is obtained by applying the maximum principle to common nonnegative smooth approximations of $\delta_0$ and passing to the limit.}
The solution of the Eq.\eqref{FP trf_h} is 
\begin{equation*}\label{lb h(tau)1}
\begin{split}
        \rho_1(\tau,y)=\int_\mathbb{R} P_1(\tau,y,0,\eta)\delta_0(\eta) d\eta=P_1(\tau,y,0,0),
\end{split}
\end{equation*}
where the above bound implies $P_1(\tau,y,0,0)\ge \dfrac{\tilde C_3}{\sqrt{\tau}}\exp{\left(-\dfrac{y^2}{\tilde C_4\tau}\right)}.$ Then, using $\tilde h(\tau,x)=e^{-\Lambda(\tau,0)}\rho_1(\tau,y)$, $y=e^{-\Lambda(\tau,0)}x$, and $e^{-\Lambda(\tau,0)}\ge C_1$,
\begin{equation}\label{lb h(tau)}
    \tilde h(\tau,x)\ge \dfrac{C_1}{\sqrt{\tau}}\exp{\left(-\dfrac{x^2}{C_2\tau}\right)},
\end{equation}
in the chosen terminal space-time window for large enough $q$.
\subsection{Proof of Theorem \ref{th_main}}
\begin{proof} Combining the lower bounds in Eqs. \eqref{lb uhat2},\eqref{lb fv*2_1}, and \eqref{lb h(tau)} into Eq. \ref{lb_z_subset},  
\begin{equation*}
    \label{lb Z} \begin{split}
        Z(T;q)&\ge\int_{T-\delta T}^T\int_{-\delta x}^{\delta x} b (-\hat{u}(t,x))f^*(t,x)h(t,x)\,dx\,dt\\&\ge\int_0^{\delta T}\int_{-\delta x}^{\delta x} C_fC_uC_1 q^{5/4} \sqrt{\tau} \exp{\left(-\dfrac{x^2}{C_2\tau}\right)} \,dx \,d\tau \\& \ge C_1  q^{5/4} \int_0^{\dfrac{1}{\Delta(q)}} \sqrt{\tau} \left(\int_{0}^{\dfrac{1}{\sqrt{\Delta(q)}}}\exp{\left(-\dfrac{x^2}{C_2\tau}\right)}dx\right)\,d\tau\\&
        \geq C_1  q^{5/4} \int_0^{\dfrac{1}{\Delta(q)}} \tau d\tau\\& \geq C_1q^{1/4},\nonumber
    \end{split}
\end{equation*}
for large enough $q$. Hence, $G'(v^*;q)>1$ for large enough $q$, which concludes the proof. \end{proof}

\subsection{Symmetry breaking and emergence of additional fixed points}\label{sec:symbreak}
\begin{theorem}
    For $q>q_0$, the map $G$ has at least two asymmetric fixed points, in addition to the symmetric fixed point $v^*=0.5$.
\end{theorem}
\begin{proof}
Consider the function $F(v)\triangleq G(v;q)-v$ for any fixed $q>q_0$. Each root of $F$ is a fixed point of $G$. It follows from Theorem \ref{th_main} that $F(v^*)=0$ and $F'(v^*)>0$. This implies for small enough $\delta>0$, $F(v^*-\delta)<0$, and $F(v^*+\delta)>0$. Since $0\leq G(v;q)\leq 1$, there are two possibilities: 
\begin{itemize}
\item{Case 1: $F(0)>0$ and $F(1)< 0$}\\
In this case, by continuity of $F$, there must exist a fixed point $v_1$ with $0<v_1<v^*-\delta$. Similar argument guarantees the existence of another fixed point $v_2$ with $v^*+\delta\leq v_2 \leq 1$. By odd symmetry of $G$, we may choose $v_2=1-v_1.$
\item{Case 2: $F(0)=0$ and $F(1)=0$}\\
In this case, $v_1=0$ and $v_2=1$ are the two additional fixed points.
\end{itemize}
This concludes the proof. 
\end{proof}
\section{Conclusions and Discussion}

We have established symmetry breaking in the one dimensional min-LQG MFG with two equivalent terminal destinations. The analysis is based on an associated scalar self-consistency map, which reduces the coupled forward-backward FP-HJB equilibrium problem to a fixed-point problem for the terminal population split. 
Our main result shows that the symmetric fixed point is stable for sufficiently weak social effect, and loses stability when the social effect is sufficiently strong. The proof proceeds by analyzing the linearization of the scalar map at the symmetric fixed point, which encapsulates the dynamics of the linearized FP-HJB system around the symmetric MFG equilibrium.

This loss of stability of the symmetric fixed point has a direct consequence for multiplicity of MFG equilibria. Using the interval structure and odd symmetry of the scalar map, we prove the existence of additional fixed points when the symmetric fixed point is unstable. These new fixed points correspond to symmetry-broken MFG equilibria in which a majority of the population selects one of the two destinations. By symmetry, such equilibria occur in pairs, representing consensus toward either destination. Thus, the model provides a rigorous example of how social conformity can destroy uniqueness and generate multiple equilibria in a non-monotone finite-horizon MFG.

The stability notion used in this paper is the fixed-point stability of the scalar map. It should not be confused with temporal stability of the forward-backward MFG system. The finite-horizon MFG equations form a coupled two-point boundary value problem in time, rather than an autonomous evolution equation for which standard stability notions apply directly. The temporal stability of \emph{stationary} MFG equilibria has been investigated in other settings \citep{yinmehmeysha12,grover2024phase,cesaroni2024stationary,hofer2025synchronization}. Extending those notions to \emph{time-varying} MFG equilibria considered here, and relating temporal stability to the fixed-point stability studied in this work, remains an interesting direction for future work.

Our proof of the existence of symmetry-broken equilibria relies on the global interval structure of the scalar map, rather than on a local bifurcation classification. A characterization of the local structure (i.e., rigorously establishing pitchfork bifurcation) near the symmetric equilibrium is therefore a natural direction for future work. Such a classification would open the door to applying singularity theory and bifurcation unfolding methods \citep{franci2022breaking,golubitsky2012singularities}. In particular, this framework could help extend our results to min-LQG MFGs with asymmetric destinations, unequal prior preferences, or heterogeneous agents, where the exact odd symmetry is broken and the symmetric deadlock equilibrium unfolds into biased consensus states, as seen in numerical simulations of \citet{salhab2019collective}.

Finally, another natural direction is to extend the present framework beyond the 1D, two-choice setting considered here. If the agents' state space is multidimensional but there are still two terminal destinations, the associated map still acts on the interval $[0,1]$; however, the FP-HJB equations defining this map will be multidimensional parabolic PDEs. A complementary extension is to allow more than two terminal choices. In that case, the terminal split becomes a probability vector in a simplex, and the scalar map is replaced by a map on this simplex. Both extensions would require additional analytical tools, but the framework developed here provides a natural starting point for such analyses.

\begin{appendices}
\renewcommand{\theequation}{\thesection.\arabic{equation}}
\section{}
\subsection{Properties of $\alpha(t,t_0)$}

\pg{\label{sec:al_prop}}
Here, we derive some bounds on $\alpha$, the fundamental solution of the mean evolution Eq.\eqref{Xj Sigma} under control $u_j$. 
Define $A(t)\triangleq   a-\dfrac{b^2}{r}p(t)$. Then, using the expression of $p(t)$ from Eq.\eqref{alpha p(t)}, we have
\begin{equation}\label{def A(t)}
    A(t)=-\Delta(q)\dfrac{1+K(q)e^{-2\Delta(q)\,(T-t)}}{1-K(q)e^{-2\Delta(q)\,(T-t)}}<0,
\end{equation}
using Eq. \ref{eq:K(q)}. Therefore, by Eq.\eqref{Xj Sigma}, the lower and upper bounds of $\alpha(\tau,t)=\exp(\bigintssss_t^\tau A(s)ds)$ for $\tau\in[t,T]$ are given by 
\begin{equation}\label{alpha ineq}
    \alpha(T,t)\le\alpha(\tau,t)\le\alpha(t,t)=1, \qquad
\end{equation}
and $\alpha(T,t)$ is bounded as
\begin{equation}\label{alphaTt ineq}
    \dfrac{1-K(q)}{e^{\Delta(q)  T}-K(q)e^{-\Delta(q) T}}\le\alpha(T,t)\le\alpha(t,t)=1, \qquad
\end{equation}
using Eq.\eqref{alpha p(t)}.

\subsection{Proof of Lemma \ref{lem:fhat}}
\pg{
\label{sec:a2}
\begin{proof}
    Recall from Section \ref{sec:Gmap} that $u^*\in U_{T}$. Now we take arbitrary control perturbation $w\in U_{T}$ with $\|w\|_{U_T}<\kappa$, and hence $u^*+w \in U_T$. Now let $f^w$ be the solution of the Fokker-Planck Eq. \ref{FP f*_v} with control $u^*+w$. Since the total drift $ax+b(u^*+w)$ grows at most linearly in $x$, the resulting density $f^w$ has a finite second moment. Furthermore, its finite second moment can be uniformly bounded in time, over the bounded perturbation ball. We denote that bound as $M(\kappa)$.
    
    The evolution of the difference $\delta f\triangleq f^w-f^*$ is given by  \begin{equation}
        \partial_t\delta f=\mathcal{L}_{v^*}(\delta f)\pg{-}\partial_x(bwf^w),\nonumber
    \end{equation}
    with $\delta f(0,x)=0$, and $L_{v^*}$ given by Eq. \ref{eq:L*}.
    Using Duhamel, \begin{equation}
    \delta f(t,x)= \pg{-}\int_0^t \mathcal{U}_*(t,s)\partial_x(bw(s,x)f^w(s,x))ds,\nonumber
     \end{equation}
     where $\mathcal{U}_*$ is the propagator associated with $\mathcal{L}_{v^*}$.
Now we recall a result of Ref. \citet{menozzi2021density}, which provides a kernel bound on the propagation of $\phi\in L^1((1+|x|)^m,\mathbb{R})$ for diffusions with rough drifts.  From Theorem 1.2(ii) of \citet{menozzi2021density}, using integration by parts, we get for $m=0,1$

\begin{align} \nonumber
    \int_{\mathbb R}(1+|x|)^m
\left|[\mathcal{U}_*(t,s)\partial_x\phi](x)\right|dx
\le \frac{C}{\sqrt{t-s}}
\int_{\mathbb R}(1+|x|)^m|\phi(x)|dx.\\
\end{align}

In our optimal control drift decomposition $A(t)x+B(x,t)$ of Sec. \ref{sec:lbf*}, we can bound $A$ (follows from result of Sec. \ref{sec:al_prop}) and $B$ (Sec. \ref{sec:lbf*}) uniformly in time. In this case, the constant $C$ in the above kernel bounds can be chosen independent of time.

Then applying this theorem to $\delta f$ for all $t < T$, 
\begin{align}
    \int_{\mathbb{R}} (1+|x|)|\delta f(t,x)|dx\leq\int_{\mathbb{R}} \int_0^t \dfrac{1}{\sqrt{t-s}}\dfrac{|w(s,x)|}{1+|x|}(1+|x|)^2f^w(s,x)dsdx\nonumber\\
    =C\int_{\mathbb{R}}\int_0^t \dfrac{1}{\sqrt{t-s}}\dfrac{|w(s,x)|}{1+|x|}(1+|x|)^2f^w(s,x)dsdx<C(\kappa)\|w\|_{U_{T}},\label{eq:deltaf}
     \end{align}
where in the last step we have used the fact that $f^w$ has a finite second moment bounded uniformly in $t$.

Next, we consider Fréchet remainder $r\triangleq f^w-f^*-\hat{f}^w$, whose evolution equation is \begin{align}
     \partial_tr=\mathcal{L}_{v^*}(r)\pg{-}\partial_x(bw\delta{f}),\nonumber
\end{align}
with $r(0,x)=0$.

Proceeding similarly as above, 
\begin{align}
\|r(t,.)\|_{L^1}=\int_{\mathbb{R}} |r(t,x)|dx<C\int_{\mathbb{R}} \int_0^t \dfrac{1}{\sqrt{t-s}}\dfrac{|w(s,x)|}{1+|x|}(1+|x|)|\delta{f}(s,x)|dsdx\nonumber\\  
< C\|w\|_{U_{T}}\int_{\mathbb{R}}\int_0^t\dfrac{1}{\sqrt{t-s}}(1+|x|)|\delta{f}(s,x)|dsdx <C(\kappa)\|w\|_{U_{T}}^2\nonumber,
\end{align}
where we have used Eq. \ref{eq:deltaf} in the last step. A computation similar to Eq. \ref{eq:deltaf} applied to $\hat{f}^w$ shows the linearized map $w\mapsto\hat f^w$ is uniformly bounded for each $T_0<T$. The above two results imply that $u\rightarrow G_2(u)$ is Fréchet differentiable as a map into $C([0,T_0];L^1)$, with derivative $w\rightarrow \hat{f}^w_{[0,T_0]}$.

Now it remains to prove that $f^w\in C([0,T];L^1)$, and that Fréchet differentiability extends to $[0,T]$. For the former, we extend the perturbed drift $ax+b(u^*+w)$ beyond $T$ while preserving its linear growth bound. The local boundedness of this perturbed drift allows the use of Bogachev–Krylov–Röckner result (Corollary 2.2.7 of \citet{bogachev2009elliptic}), which along with Scheffé’s lemma \citep{resnick1999probability} establishes continuity on $(0,T]$. Standard heat kernel bounds extend it to $[0,T]$. 

Finally, the uniform remainder estimate proved above implies that $(f^{\epsilon w}-f^*)/\epsilon$ is uniformly Cauchy on $[0,T]$, hence extending $\hat{f}^w$ continuously to T, and yielding \\$\sup_{0\leq t\leq T}\|f^w(t)-f^*(t)-\hat{f}^w(t)\|_{L^1}\leq C(\kappa)\|w\|^2_{U_T}$. The Fréchet differentiability on  $[0,T]$ then follows from the boundedness of the linear map.
\end{proof}}

\end{appendices}

\section*{Competing Interests} The authors do not have any competing interests.
\section*{Acknowledgments}\pg{
This material is based upon work supported by the US National Science Foundation under Grant No. 2102112. The authors used ChatGPT (OpenAI, GPT-5/5.2/5.6) as a tool for language editing and literature search, where it helped locate Refs. \citet{menozzi2021density} and \citet{bogachev2009elliptic} used in proof of lemma 4.3.}
\clearpage

\bibliography{refs_paper2}
\end{document}